\documentclass[12pt]{amsart}
\usepackage{amssymb}
\usepackage{stmaryrd}
\usepackage{mathabx}
\usepackage{mathptmx}
\usepackage{cleveref}
\usepackage[all]{xy}
\usepackage{color}
\usepackage{tikz}
\usepackage{todonotes}
\usepackage{paralist}
\numberwithin{equation}{section}

\newcommand{\Z}{\mathbb{Z}}
\newcommand{\K}{\mathbb{K}}

\newcommand{\N}{\mathbb{N}}

\DeclareMathOperator{\Ann}{Ann}

\DeclareMathOperator{\gin}{gin}

\DeclareMathOperator{\link}{lk}
\DeclareMathOperator{\lk}{lk_{\Delta} }
\DeclareMathOperator{\Mon}{Mon}

\DeclareMathOperator{\Tor}{Tor}

\DeclareMathOperator{\pnt}{\raise 0.5mm \hbox{\large\bf.}}

\let\phi=\varphi

\newtheorem{theorem}{Theorem}[section]
\newtheorem{lemma}[theorem]{Lemma}
\newtheorem{proposition}[theorem]{Proposition}
\newtheorem{corollary}[theorem]{Corollary}

\newtheorem{lem-def}[theorem]{Lemma and Definition}
\newtheorem{prop-def}[theorem]{Proposition and Definition}

\theoremstyle{definition}
\newtheorem{definition}[theorem]{Definition} 
\newtheorem{remark}[theorem]{Remark}
\newtheorem{rem-def}[theorem]{Remark and Definition}
\newtheorem{example}[theorem]{Example}

\newtheorem{question}[theorem]{Question}
\newtheorem{problem}[theorem]{Problem}

\title{Squeezed Complexes}

\author{Martina Juhnke-Kubitzke}
\address{Universit\"at Osnabr\"uck, Institut f\"ur Mathematik, 49069 Osnabr\"uck, Germany}
\email{juhnke-kubitzke@uos.de}

\author[Uwe Nagel]{Uwe Nagel}
\address{Department of Mathematics, University of Kentucky, 715 Patterson Office Tower, Lexington, KY 40506-0027, USA}
\email{uwe.nagel@uky.edu}

\begin{document}

\begin{abstract} 
Given a shifted order ideal $U$, we associate to it a family of simplicial complexes $(\Delta_t(U))_{t\geq 0}$ that we call squeezed complexes. In a special case, our construction gives squeezed balls that were defined and used by Kalai to show that there are many more simplicial spheres than boundaries of simplicial polytopes. We study combinatorial and algebraic properties of squeezed complexes. In particular, we show that they are vertex decomposable and characterize when they have the weak or the strong Lefschetz property. Moreover, we define a new combinatorial invariant of pure simplicial complexes, called the singularity index, that can be interpreted as a measure of how far a given simplicial complex is from being a manifold. In the case of squeezed complexes $(\Delta_t(U))_{t\geq 0}$, the singularity index turns out to be strictly decreasing until it reaches (and stays) zero if $t$ grows.
\end{abstract}

\date{August 1, 2018}

\thanks{The first author was supported by the German Research Council DFG GRK-1916. The second author was partially supported by Simons Foundation grant \#317096.  }

\maketitle



\section{Introduction}
\label{sect_intro}

In a landmark paper \cite{Ka}, Kalai introduced a large class of triangulated balls, named squeezed balls, with remarkable properties. They are vertex decomposable, thus shellable, and their boundary complexes amount to many more triangulated spheres than there are boundary complexes of simplicial polytopes. These complexes were studied further by Murai \cite{Mu-07, Mu-10}. In this paper we introduce and study an extension of Kalai's construction. 

A shifted order ideal of a polynomial ring $P(m) = \K[x_1,\ldots,x_m]$ in $m$ variables over a field $\K$ is a set $U \subset P(m)$ of monomials satisfying the following conditions: 
\begin{itemize}
\item[(i)] if a monomial divides a monomial in $U$ then this monomial belongs to $U$; 

\item[(ii)] any variable dividing a monomial in $U$ can be replaced by any variable with a \emph{larger} index to obtain another monomial in $U$; 

\item[(iii)] $1, x_1,\ldots,x_m \in U$. 
\end{itemize}
In general, Condition (iii) is not assumed. However, it is convenient (and harmless). 

Let $U \subset P(m)$ be a finite shifted order ideal. Denote by $d_{\max} (U)$ the maximum degree of a monomial in $U$. For every integer $t \ge 0$, we construct a pure simplicial complex 
$\Delta_t (U)$ on the vertex set $\{1,2,\ldots,m+d_{\max}(U) + t\}$ of dimension 
$d_{\max} (U) + t-1$ (see Definition~\ref{def:squeezed complex}). For $t \ge d_{\max} (U)$ we recover the squeezed balls, as constructed by Kalai. Thus, we call $\Delta_t (U)$ a squeezed 
complex of index $d_{\max} (U) - t$ or simply a $(d_{\max} (U) - t)$-squeezed complex. An arbitrary 
squeezed complex is not necessarily a manifold, not even a pseudomanifold. However, we show that squeezed complexes  share several properties with squeezed balls. Every squeezed complex is vertex decomposable and thus shellable (see Theorem~\ref{thm:vertexdecomposable}). 
Moreover, fixing $U$, but varying $t \ge 0$, all squeezed complexes $\Delta_t (U)$ have the same $h$-vector (up to adding zeros at the end). Namely, $h_i(\Delta_t(U)) $ counts the number of monomials in $U$ of degree $i$ for $0\leq u\leq d_{\max}(U)$ and is zero otherwise. Extending results by Murai \cite{Mu-07, Mu-10}, we 
explicitly describe the Stanley-Reisner ideal of any squeezed complex 
(see Theorem~\ref{thm:SR-ideal})  as well as its generic initial ideal (see Proposition~\ref{prop:gin squeezed}). This also allows us to discuss Lefschetz properties. These important properties have been formalized in \cite{HMNW}, and we refer to the monograph \cite{HMMNWW} for further information. 

As conjectured by Kalai \cite{Ka} and shown by Murai \cite{Mu-07, Mu-10}, any boundary complex of a squeezed ball $\Delta_t (U)$ with $t \ge d_{\max} (U)$ has the strong Lefschetz property. However, this is not necessarily true for the squeezed ball itself. We characterize when an arbitrary squeezed complex has the weak or strong Lefschetz property, respectively (see Propositions~\ref{prop:char WLP} and \ref{prop:char SLP}). 

As mentioned above, an arbitrary squeezed complex $\Delta_t (U)$ is not even a pseudomanifold, but it is always a manifold, even a ball, if $t \ge d_{\max} (U)$. In fact, it is easy to see that for $m\geq 2$ and $d_{\max}(U)\geq 2$ a squeezed complex $\Delta_t(U)$ is a pseudomanifold if and only if it is a ball. Moreover, if $U$ is a maximal shifted order ideal, i.e., $U \subset P(m)$ contains every monomial in $U$ whose degree is at most $d_{\max} (U)$, then either of these conditions is true precisely if $t \ge d_{\max} (U)$ (see Theorem~\ref{thm:max order manifold}). We confirm the intuition that, for an arbitrary finite shifted order ideal  $U$, the squeezed complex $\Delta_t (U)$ is the closer to a pseudomanifold the larger $t$ is. To this end we associate to any pure simplicial complex $\Delta$ a quantity, called singularity index $D (\Delta)$ (see Definition~\ref{def:sing index}). It is a non-negative rational number that can be seen as a measure of how far $\Delta$ is from being a pseudomanifold. More precisely, if $\Delta$ is strongly connected, then its singularity index is zero if and only if $\Delta$ is a pseudomanifold. Fixing a finite shifted order ideal $U$, we show that the sequence of singularity indices $(D (\Delta_t (U)))_{t \ge 0}$ is strictly decreasing until it first becomes zero (see Theorem~\ref{thm:sing index}). This can be interpreted as replacing a given simplicial complex that is not a pseudomanifold by a sequence of simplicial complexes of increasing dimensions with  the same $h$-vector (up to additional zeros) that eventually are pseudomanifolds. Thus, it is somewhat reminiscent to the process of resolving singularities of a variety by repeated blow-ups  in algebraic geometry. 

Note that every shifted pure simplicial complex is up to coning a squeezed complex $\Delta_0 (U)$ for some order ideal $U$ (see Corollary~\ref{cor:shifted are squeezed}). Thus, the mentioned process may be viewed as a systematic way of replacing an arbitrary shifted pure simplicial complex by a pseudomanifold with the same $h$-vector (up to additional zeros). It would be interesting to extend this procedure to situations in which one starts with more general simplicial complexes such as, e.g., non-pure shifted simplicial complexes. 

Now we briefly describe the organization of this note. In the following section we  review some basic concepts and results.   Squeezed complexes are defined  in Section~\ref{sec:squeezed}.   There we also establish that they are always 
vertex-decomposable. Moreover, we describe explicitly the Stanley-Reisner ideal of any squeezed complex. These results are 
used in Section~\ref{sect: Alg Prop} in order to show that the graded Betti numbers of a squeezed complex can be read off from its shifted order ideal and to characterize whether it has a   Lefschetz property.  In Section~\ref{sec:sing index} we introduce  the singularity index of an arbitrary  pure simplical complex and determine it for some squeezed complexes. We conclude by establishing the mentioned property of the sequence of complexes $(\Delta_t (U))_{t \ge 0}$. The final section is devoted to a discussion of some questions that arose from this work.


\section{Preliminaries} 
    \label{sec:prelim}

In this section, we will provide the necessary background knowledge for this article. This includes some basics on simplicial complexes as well as some facts on strongly stable ideals and shifted order ideals and the relation between those two classes of objects. 

We start by fixing some notation that will be used. Throughout this article, we let $\K$ be an arbitrary field and we set $[n]: = \{1,2,\ldots,n\}$ for any $n\in \N$. Moreover, for $a,b\in\N$, we let $[a,b]$ denote the set of integers between $a$ and $b$, i.e., $[a,b]:=\{c\in \N~|~a\leq c\leq b\}$. For a finite set $A\subset \N$ and $n\in \N$, we set $A+n:=\{a+n~|~a\in A\}$. 

\subsection{Combinatorics of simplicial complexes}
We recall basic facts on simplicial complexes, including some of their combinatorial properties. We refer to \cite{Stanley-greenBook} for more details.

Given a finite set $V$, an (abstract) \emph{simplicial complex} $\Delta$ on the vertex set $V$ is a collection of subsets of $V$ that is closed under inclusion. Throughout this paper, all simplicial complexes are assumed to be finite. 
Elements of $\Delta$ are called \emph{faces} of $\Delta$ 
and inclusion-maximal faces are called \emph{facets} of $\Delta$. The \emph{dimension} of a face $F \in \Delta$ is its cardinality minus one, and the \emph{dimension} of $\Delta$ is defined as $\dim\Delta:=\max\{\dim F~|~F\in \Delta\}$. $0$-dimensional and $1$-dimensional faces are called \emph{vertices} and \emph{edges}, respectively. 
A \emph{ridge} is a face of dimension $\dim \Delta - 1$. A simplicial complex $\Delta$ is \emph{pure} if all its facets have the same dimension. The \emph{link} of a face $F\in \Delta$ is the subcomplex
$$
\lk(F)=\{G\in\Delta~|~G\cap F=\emptyset,\; G\cup F\in \Delta\}.
$$
We will write $\lk(v)$ for the link of a vertex $v$. 
The \emph{deletion} $\Delta\setminus F$ of a face $F\in \Delta$ from $\Delta$ describes the simplicial complex $\Delta$ outside of $F$:
$$
\Delta\setminus F=\{G\in \Delta~|~ F\not\subseteq G\}.
$$
If $F=\{v\}$ is a vertex, then we write $\Delta\setminus v$ instead of $\Delta\setminus\{v\}$ for short. 
For subsets $F_1,\ldots,F_m$ of a finite set $V$, we denote by $\langle F_1,\ldots,F_m\rangle$ the smallest simplicial complex containing $F_1,\ldots,F_m$, i.e.,
$$
\langle F_1,\ldots,F_m\rangle=\{G\subseteq V~|~G\subseteq F_i\mbox{ for some }1\leq i\leq m\}.
$$
Given a $(d-1)$-dimensional simplicial complex $\Delta$, we let $f_i(\Delta)$ denote the number of $i$-dimensional faces of $\Delta$ and we call $f(\Delta)=(f_{-1}(\Delta), f_0(\Delta), \ldots, f_{d-1}(\Delta))$ the \emph{$f$-vector} of $\Delta$. The \emph{$h$-vector} $h(\Delta)=(h_0(\Delta),h_1(\Delta),\ldots,h_d(\Delta))$ of $\Delta$ is defined via the relation
$$
\sum_{i=0}^{d}f_{i-1}(\Delta)(t-1)^{d-i}=\sum_{i=0}^dh_i(\Delta)t^{d-i}.
$$

We now consider several relevant combinatorial properties a simplicial complex.

\begin{definition}
Let $\Delta$ be a pure simplicial complex.
\begin{itemize}
\item[(i)] $\Delta$ is called \emph{shellable} if there exists an ordering $F_1,\ldots,F_n$ of the facets of $\Delta$ such that for each $1\leq i\leq n$ there exists a unique minimal element $R_i$ in 
$$\{G\subset F_i~|~G\not\subseteq F_j\mbox{ for } 1\leq j\leq i-1\}.$$
\item[(ii)] $\Delta$ is called \emph{vertex-decomposable} if either $\Delta=\{\emptyset\}$ or there exists a vertex $v\in \Delta$, a so-called \emph{shedding vertex}, such that $\lk(v)$ and $\Delta\setminus v$ are vertex-decomposable. 
\end{itemize}
\end{definition}

If $\Delta$ is a $(d-1)$-dimensional vertex decomposable simplicial complex with shedding vertex $v$, then the face numbers of $\Delta$ can be easily computed via the following recursion:
$$
f_i(\Delta)=f_i(\Delta\setminus v)+f_{i-1}(\lk(v)) \quad \mbox{for}\quad 0\leq i\leq d-1.
$$
Using the relations between $f$-vectors and $h$-vectors this immediately translates into:
\begin{equation}\label{h-vector:vertexdecomposable}
h_i(\Delta)=h_i(\Delta\setminus v)+h_{i-1}(\lk(v)) \quad \mbox{for}\quad 0\leq i\leq d.
\end{equation}
It is well-known \cite{PB} that vertex-decomposability is a stronger property than shellability. Moreover, every shellable simplicial complex has the following property: 

\begin{definition}
A pure $(d-1)$-dimensional simplicial complex $\Delta$ is called \emph{strongly connected} if for any two facets $F$ and $G$ of $\Delta$, there exist facets $F=F_0,F_1,\ldots, F_m=G$ such that $|F_i\cap F_{i+1}|=d-1$ for all $0\leq i\leq m-1$. 
\end{definition}

In this article, we will be interested in yet another class of simplicial complexes that are strongly connected but not necessarily shellable. 

\begin{definition}
 A pure simplicial complex $\Delta$ is a \emph{pseudomanifold} if it is strongly connected and every ridge lies in at most two facets. 
\end{definition} 

Both, shellable simplicial complexes and pseudomanifolds are strongly connected but there exist shellable simplicial complexes that are not pseudomanifolds and, vice versa, not every pseudomanifold is shellable. 

A pure $(d-1)$-dimensional simplicial complex $\Delta$ is a \emph{(simplicial) sphere} respectively \emph{ball} if the geometric realization of $\Delta$ is homeomorphic to the $(d-1)$-sphere $\mathbb{S}^{d-1}$ respectively the $(d-1)$-ball $\mathbb{B}^{d-1}$. 
Clearly, every ball or sphere is a pseudomanifold, but the class of pseudomanifolds contains far more objects than just balls and spheres. However, if we require a pseudomanifold to be shellable, then the following result is true (see e.g., \cite{DK}):
\begin{lemma}
    \label{lem:Kalai criterium} 
Let $\Delta$ be a shellable simplicial complex. Then $\Delta$ is a pseudomanifold if and only if $\Delta$ is a ball or a sphere. 
\end{lemma}

\subsection{Shifted order ideals and strongly stable ideals}
We start by fixing some notation. 
We write $P(m) = \K[x_1,\ldots,x_m]$ for a polynomial ring in $m$ variables over a field $\K$.  
Moreover, $\Mon(m)$ denotes the set of all monomials in $P(m)$, and $\Mon_{\leq k}(m)$ and $\Mon_k(m)$ the set of those monomials in $P(m)$ whose degree is at most $k$ and is equal to $k$, respectively. 
 For a monomial ideal $I\subset P(m)$, we use $G(I)$ to denote its  unique minimal set of monomial generators. We write $d_{\max}(I)$ for  the maximal degree of an element in $G(I)$.  
A monomial ideal $I\subset P(m)$ is called \emph{strongly stable} if for each $x_i u\in I$ also the monomial $x_j u$ lies in $I$ whenever $1\leq j<i$, that is, any variable dividing a monomial in $I$ can be replaced by any variable with a \emph{smaller} index to obtain another monomial in $I$.  Note that it is enough to verify this condition for the minimal monomial generators of $I$. 

A (finite) \emph{monomial order ideal} of $P(m)$ is a (finite) subset $U$ of $\Mon (m)$, that is closed under taking divisors, i.e., whenever $u \in U$ and $v$ divides $u$, then $v \in U$. Observe that every non-empty monomial 
order ideal contains the monomial $1$. In the following, we will always assume that a monomial order ideal $U\subset P(m)$ contains all the variables $x_1,\ldots,x_m$. Moreover, without further mentioning, all monomial order ideals appearing in this article will be finite. We denote with $d_{\max}(U)$ the maximal degree of a monomial in $U$ and with $[U]_k$ the set of degree $k$  monomials of $U$.  
A monomial order ideal $U$ is said to be \emph{shifted} if with each monomial $x_i u\in U$ also the monomial $x_j u$ is in $U$ whenever $i < j \le m$, that is, any 
variable dividing a monomial in $U$ can be replaced by any variable with a \emph{larger} index to obtain another monomial in $U$. 

It is easy to see that Artinian monomial ideals in $P(m)$ and finite order ideals in $P(m)$ are in one-to-one-correspondence. Moreover, under this correspondence, strongly stable ideals and shifted order ideals are mapped to each other. We now make this construction precise:

\begin{definition} 
     \label{def:I(U)}
Let $U \subset \Mon (m)$ be a monomial order ideal.  The \emph{ideal to $U$} is the monomial ideal $I(U)$ of $P(m)$ that is generated by the monomials in 
$\Mon (m) \setminus U$. 
\end{definition} 

If $U$ is a finite monomial order ideal, then every monomial in $P(m)$ of sufficiently large degree belongs to $I(U)$. Thus, the quotient ring $P(m)/I(U)$ is a finite-dimensional $\K$-vector space, or, in other words, $I(U)$ is an Artinian ideal. 

We will now compare the maximal degrees $d_{\max}(U)$ and $d_{\max}(I(U))$  of $U$ and $I(U)$, respectively.  First observe that, since $U$ contains all the variables, all elements of $I(U)$ are of degree at least $2$ and hence $d_{\max}(I(U))\geq 2$. 
If $U \subset \Mon_{\le d+1} (m)$ and thus $d_{\max}(U)\leq d+1$, then we have  
\[
I(U) = (u~|~u \in\Mon_{\le d+2} (m) \setminus U). 
\]
In particular, this means that $d_{\max}(I(U))\leq d+2\leq d_{\max}(U)+1$.  We can also classify, when equality holds.

\begin{lemma} 
    \label{lem:shifted vs strongly stable}
Let $U\subset P(m)$ be a monomial order ideal. Then:
\begin{itemize}
\item[(i)] $2\leq d_{\max}(I(U))\leq d_{\max}(U)+1$.
\item[(ii)] $I(U)$ is strongly stable if and only if $U$ is shifted. In this case,
$$
d_{\max}(I(U))=d_{\max}(U)+1.
$$
\end{itemize}
\end{lemma}

\begin{proof}
(i) directly follows from the discussion preceding this lemma.

We now show (ii). Assume that $I(U)$ is strongly stable. Let $x_iu\in U$ and $i<j\leq m$. If $x_ju\notin U$, then $x_ju\in I(U)$ and, as $I(U)$ is strongly stable, we would also have $x_iu\in I(U)$, which is a contradiction. Hence, $U$ is shifted. The other direction follows from the same arguments. 

To show $d_{\max}(I(U))=d_{\max}(U)+1$, by (i), it suffices to prove that $d_{\max}(I(U))\geq d_{\max}(U)+1$. Let $d_{\max}(U)=d$. As $I(U)=\Mon(m)\setminus U$ and $U\subset \Mon_{\leq d}(m)$, the ideal $I(U)$ contains all monomials in $P(m)$ of degree $d+1$. In particular, we have $x_m^{d+1}\in I(U)$. Now assume that $d_{\max}(I(U))\leq d$. As $x_m^{d+1}$ has to be divisible by a minimal generator of $I(U)$, we deduce that $x_m^\ell\in G(I(U))$ for some $0\leq \ell\leq d$. As $I(U)$ is strongly stable, this forces $I(U)$ to contain all monomials in $P(m)$ of degree $\ell$. By definition of $I(U)$ it follows that $U\subset \Mon_{\ell-1}(m)$, which implies $d_{\max}(U)< \ell\leq d$. We hence arrive at a contradiction and the claim follows.
\end{proof}

The next example shows that in (i) $d_{\max}(I(U))$ can indeed take any possible value.
\begin{example}
Let $1\leq \ell\leq m$ and let $U^{(\ell)}\subset P(m)$ be the union of  all monomials of degree $\leq \ell$ and all squarefree monomials in $P(m)$. It is straightforward to check that $U^{(\ell)}$ is an order ideal, which is not shifted, and that $d_{\max}(U^{(\ell)})=m$. 
The ideal $I(U^{(\ell)})$ to $U^{(\ell)}$ consists of all non-squarefree monomials of degree at least $\ell+1$. 
As every non-squarefree monomial of degree $\ell+2$ is divisible by a non-squarefree monomial of degree $\ell+1$, we have
$$
G(I(U^{(\ell)}))=\{ u\in \Mon_{\ell+1}(m)~|~u\mbox{ is not squarefree}\},
$$
from which it follows that $d_{\max}(I(U^{(\ell)})=\ell+1$.
\end{example}

\Cref{def:I(U)} has an algebraic interpretation using Macaulay-Matlis duality. To this end let $R(m) = \K[y_1,\ldots,y_m]$ be a polynomial ring in dual variables $y_1,\ldots,y_m$. 
For simplicity, assume temporarily that $\K$ has characteristic zero. Interpreting the elements of $R(m)$ as differential operators, there are products $[R(m)]_i \times [P(m)]_j \to [P(m)]_{j-i},\ (f, g) \mapsto f \cdot g$, that turn $P(m)$ into a graded $R(m)$-module. (For a construction in positive characteristic, see \cite{IK}.) For  a homogeneous ideal $I \subset R(m)$,   we define its \emph{inverse system $I^\perp$} as 
\[
I^\perp = \{ g \in P(m) \; | \; f \cdot g = 0 \text{ for all } f \in I\}. 
\]
This is a graded $R(m)$-submodule of $P(m)$. Conversely, if $M$ is a graded $R(m)$-submodule of $P(m)$, then we define its \emph{annihilator} $\Ann (M)$ as 
\[
\Ann (M) = \{ f \in R(m) \; | \; f \cdot g \text{ for all } g \in M\}. 
\]
$\Ann(M)$ is a homogeneous ideal of $R(m)$. In fact, both constructions are inverse to each other (see, e.g., \cite{IK}). 

\begin{theorem}
    \label{thm:Macaulay duality} 
There is a bijection between the set of homogeneous ideals $I   \subset R(m)$ and the set of graded $R(m)$-submodules $M$ of $P(m)$ induced by $I \mapsto I^\perp$ and $M  \mapsto \Ann (M)$.   Moreover, $I^\perp$ is a finitely generated $R(m)$-module if and only of $R(m)/I$ is Artinian. 
\end{theorem} 

Since, for every integer $j$,   the paring $[R(m)]_j \times [P(m)]_j \to \K$ is exact one has $\dim_\K [R(m)/I]_j = \dim_\K [I^\perp]_j$. 

If  $I$ is a monomial ideal, then its inverse system $I^\perp$ has a $\K$-basis consisting of monomials, and this basis is an order ideal of $P(m)$. 
We now simplify notation by identifying $R(m)$ and $P(m)$. Thus, the inverse system $I^{\perp}$ of a monomial ideal $I \subset P(m)$ is generated as $\K$-vector space by the monomial order ideal 
 $\Mon (P(m)) \setminus I$. We set $U(I):=\Mon (m) \setminus I$. It follows that $\dim_\K [P(m)/I]_j = \# [U(I)]_j$.

The following observation relates these considerations to Definition~\ref{def:I(U)}. 

\begin{lemma}
If $U$ is an order ideal of $P(m)$, then $I(U) = \Ann (U)$. 
\end{lemma} 

\begin{proof}
Consider any monomial $x^a=x_1^{a_1}\cdots x_m^{a_m} \in P(m)$. Then $x^a$ is in $\Ann (U)$ if and only if, for every monomial $x^b=x_1^{b_1}\cdots x_m^{b_m} \in U$, there is some exponent $a_i$ with $a_i > b_i$, that is, $x^a$ does not divide any monomial in $U$. Since $U$ is a monomial order ideal, the latter is equivalent to $x^a \notin U$. 
\end{proof} 

\begin{corollary}
If $U$ is an order ideal of $P(m)$, then
 $$U = \Mon (m) \setminus I(U)=U(I(U))$$
  and  $\dim_\K [P(m)/I(U)]_j = \# [U]_j$ for all $j$. 
\end{corollary} 

\section{Squeezed complexes}
    \label{sec:squeezed}

In this section, given a shifted order ideal $U\subset P(m)$, we associate to it a family of simplicial complexes $(\Delta_t(U))_{t\geq 0}$, that will be called \emph{squeezed complexes} of index $d_{\max}(U)-t$. For the case $t\geq d_{\max}(U)$, we obtain squeezed balls, as introduced by Kalai \cite{Ka} and further studied by Murai \cite{Mu-07,Mu-10}. We will be interested in combinatorial properties of those complexes in \Cref{subsect:Comb squeezed} and study their Stanley-Reisner ideals in \Cref{subsect:squeezed and stable}.

\subsection{Combinatorics of squeezed complexes}\label{subsect:Comb squeezed}
In order to define squeezed complexes, we need to introduce several further notions. 

Let $t\geq 0$ be a non-negative integer and let $U\subset P(m)$ be an order ideal. Let $d_{\max }:=d_{\max}(U)$ be the maximal degree of a monomial in  $U$. 
For each $u=x_{i_1}\cdots x_{i_s}\in U$ with $i_1\leq \cdots \leq i_s$, we define a $(d_{\max}+t)$-subset $F_t(u)\subseteq [m+t+d_{\max}]$:
\begin{itemize} 
\item[$\triangleright$] For $s\leq t$, we set
\begin{equation}\label{eq:s<t}
F_t(u):=\bigcup_{\ell=1}^s\{i_\ell+2(\ell-1),i_\ell+2\ell-1\}\cup [m+2s+1,m+t+d_{\max}].
\end{equation}
\item[$\triangleright$] For $s \ge t$, we set
\begin{equation}\label{eq:s>t}
 F_t(u):=\bigcup_{\ell=1}^t\{i_\ell+2(\ell-1),i_\ell+2\ell-1\}\cup \bigcup_{\ell=t+1}^s\{i_\ell+\ell+t-1\}\cup [m+t+s+1,m+t+d_{\max}].
\end{equation}
\end{itemize}
We note that for $s=t$ \eqref{eq:s<t} and \eqref{eq:s>t} yield the same set. Observe that the set $[m+t+s+1,m+t+d_{\max}]$ is empty if and only if $s = d_{\max}$. Similarly, $ [m+2s+1,m+t+d_{\max}] = \emptyset$ if and only if $s = t = d_{\max}$.

It is easy to check that the above definition gives rise to an injection
 $$
 \Mon_{\leq d_{\max}}(m) \to \{A\subseteq [m+t+d_{\max}]~|~\#A=d_{\max}+t\}: \ u \mapsto F_t (u).$$
  We omit a proof of this fact since this also follows from \Cref{cor:hVector} combined with the relation $f_{d-1}(\Delta)=\sum_{i=0}^d h_i(\Delta)$ for a $(d-1)$-dimensional simplicial complex $\Delta$. 

Using the previous construction, we can provide our main definition:  

\begin{definition} 
     \label{def:squeezed complex}
Let $t\geq 0$ and let $U\subset P(m)$ be a shifted order ideal. Let $\Delta_t(U)$ be the $(d_{\max}(U)+t-1)$-dimensional simplicial complex on vertex set $[m+t+d_{\max}(U)]$ , whose facets are given by $F_t(u)$ for $u\in U$. We call $\Delta_t(U)$ the \emph{squeezed complex of index $d_{\max}(U)-t$} associated to $U$. We also say that $\Delta_t(U)$ is the \emph{$(d_{\max}(U)-t)$-squeezed complex} of $U$.
\end{definition}

\begin{remark}
It is immediate from the construction that the squeezed complexes of \emph{non-positive index} (i.e., $t\geq d_{\max}(U)$) are exactly the squeezed balls, introduced by Kalai \cite{Ka}. This also motivates the naming of those complexes. The larger the index of a squeezed complex, the smaller its dimension and its vertex set and, intuitively, the more squeezed the complex is. 
\end{remark}

Though, a priori, we associate an infinite family of simplicial complexes to a shifted order ideal, it will follow from the next lemma, that in order to understand the whole family it suffices to understand only finitely many of its members. 
\begin{lemma}
Let $U\subset P(m)$ be a shifted order ideal and let $t> d_{\max}(U)$. Then
$$
\Delta_t(U)=\{m+t+d_{\max}(U)+1\}\ast \Delta_{t-1}(U).
$$
In particular, $\Delta_t(U)$ is the join of $\Delta_{d_{\max(U)}}$ with the $(t-d_{\max}(U)-1) $-simplex on vertex set $\{m+2d_{\max}(U)+1,\ldots,m+t+d_{\max}(U)\}$. 
\end{lemma}

\begin{proof}
First observe, that, since $t>d_{\max}(U)$, we use \eqref{eq:s<t} for the construction of $F_t(u)$ for any $u\in U$. The claim then directly follows since $[m+2d_{\max}(U)+1,m+t+d_{\max}(U)]\subseteq F_t(u)$ for any $u\in U$.
\end{proof}

\begin{remark}
By the previous lemma, the family of squeezed complexes $(\Delta_t(U))_{t\geq 0}$ associated to a shifted order ideal $U\subset P(m)$, is completely determined by those squeezed complexes $\Delta_t(U)$ with $0\leq t\leq d_{\max}(U)$, i.e., those with non-negative index. In particular, given $U$, its squeezed complexes, in the sense of Kalai, only differ in additional cone points. In this sense, any shifted order ideal $U$ has a \emph{unique} minimal squeezed complex, in the sense of Kalai, namely $\Delta_{d_{\max} (U)} (U)$. 
\end{remark}

Let us consider some extremal cases. 

\begin{example}
    \label{exa:m=1}
(i) Any order ideal $U \subset P(1)$ with at least two elements is shifted. Setting $d = d_{\max} (U)$, the $(d-t)$-squeezed complex $\Delta_t (U)$ with $0 \le t \le d$ is the join of the $(t-1)$-simplex on the vertex set $\{2, 4,\ldots,2t\}$ and the boundary of the $d$-simplex on the vertex set $[d+t+1] \setminus \{2, 4,\ldots,2t\}$. Note that for $t=0$, $\{2, 4,\ldots,2t\}=\emptyset$ and one thus only gets the boundary of the $d$-simplex Indeed, one computes 
\[
F_t (x_1^s) = \begin{cases}
[2s] \cup [2s+2, d+t+1], & \text{ if } s \le t \\
[t+s] \cup [t+s+2, d+t+1], & \text{ if } s \ge t. 
\end{cases}
\]
It follows that $\Delta_t (U)$ is a $(d+t-1)$-dimensional pseudomanifold. 
Since $\Delta_t(U)$ is also shellable (see Theorem~\ref{thm:vertexdecomposable} below), it is even a shellable ball or sphere  by Lemma~\ref{lem:Kalai criterium}. Looking at the homology of $\Delta_t(U)$, we see that it is a sphere if and only if $t=0$ and a ball otherwise.

(ii) Consider the shifted order ideal  $U = \{1, x_1,\ldots,x_m\} \subset P(m)$. Then $\Delta_0 (U)$ is a $0$-dimensional simplicial complex with $m+1$ vertices, whereas $\Delta_1 (U)$ is a $1$-dimensional ball on $[m+2]$. Its facets are $\{i, i+1\}$ for  $i = 1,2,\ldots,m+1$. 
\end{example}

It was shown in \cite{Ka} that Kalai's squeezed complexes are vertex decomposable and that their $h$-vectors can be read off directly from the underlying order ideal. In the following, we will show that both properties remain true for the more general $t$-squeezed complexes. 

We first need some preparations:

\begin{lemma}\label{lemma:link}
Let $U\subset P(m)$ be a shifted order ideal and let $\widetilde{U}=\{u\in U~|~x_1\cdot u\in U\}$. Then:
\begin{itemize}
\item[(i)] $\link_{\Delta_0(U)}(1)=\langle F_0(u)+1~|~u\in \widetilde{U}\rangle$. In particular, $\link_{\Delta_0(U)}(1)$ is isomorphic to $\Delta_0(\widetilde{U})$. 
\item[(ii)] If $1\leq t\leq d_{\max}(U)$, then
$$
\link_{\Delta_t(U)}(1)=\{2\}\ast\langle F_{t-1}(u)+2~|~u\in \widetilde{U}\rangle.
$$
In particular, $\link_{\Delta_t(U)}(1)$ is isomorphic to a cone over the squeezed complex $\Delta_{t-1}(\widetilde{U})$. 
\end{itemize}
\end{lemma}

\begin{proof}
First note that, since $\Delta_t(U)$ is a pure simplicial complex, so is $\link_{\Delta_t(U)}(1)$. A facet $F_t(u)\in \Delta_t(U)$ contains the vertex $1$ if and only if $x_1$ divides $u$, i.e., $u=x_1\cdot \tilde{u}$ for some $\tilde{u}\in \Mon(m)$. As $U$ is an order ideal, we have $\tilde{u}\in U$. It hence follows that
$$
\link_{\Delta_t(U)}(1)=\langle F_t(x_1\cdot u)\setminus \{1\}~|~ u\in\widetilde{U}\rangle.
$$
For $t=0$ all facets of $\Delta_t(U)$ are computed via \eqref{eq:s>t} and every variable of a monomial $u\in U$ is represented by exactly one vertex in the facet $F_t(u)$. The claim of (i) now follows immediately from \eqref{eq:s>t}. 
For $1\leq t\leq d_{\max}(U)$, the situation is slightly more complicated. Now, depending on the degree of a monomial $u\in U$ either \eqref{eq:s<t} or \eqref{eq:s>t} are used for the construction of $F_t(u)$. In both cases, we have that if $1\in F_t(u)$, then also $2\in F_t(u)$, since the first occurence of the variable $x_1$ in $u$ gives rise to the two vertices $1$ and $2$ in $F_t(u)$. Therefore,
$$
\link_{\Delta_t(U)}(1)=\{2\}\ast \langle F_t(x_1\cdot u)\setminus \{1,2\}~|~ u\in \widetilde{U}\rangle.
$$
As $F_t(x_1\cdot u)=F_{t-1}(u)+2$ for $ u\in \widetilde{U}$, the claim follows.
\end{proof}

We want to remark that the essence of the previous lemma is that~--~up to additional cone points and isomorphism~--~the link of the vertex $1$ in a squeezed complex is again a squeezed complex (of possibly larger index).

A similar statement is true for the deletion of the vertex $1$. 

\begin{lemma}\label{lemma:deletion}
Let $m\geq 2$. Let $U\subset P(m)$ be a shifted order ideal and let $\hat{U}=\{u\in U~|~x_1\not\mid u\}$. Then,
$$
\Delta_t(U)\setminus 1=\langle F_t(u)~|~u\in\hat{U}\rangle.
$$
In particular, $\Delta_t(U)\setminus 1$ is isomorphic to the $(d_{\max}(\hat{U})-t)$-squeezed complex of $\hat{U}$ (considered as order ideal in $\K[x_2,\ldots,x_m]$).
\end{lemma}

Note that in the above lemma, we have $d_{\max}(U)=d_{\max}(\hat{U})$: Indeed, as $U$ is shifted, we have $x_m^{d_{\max}(U)}\in U$ and hence, as $m\geq 2$, also $x_m^{d_{\max}(U)}\in \hat{U}$. 

\begin{proof} 
If $u\in \hat{U}$, then $x_1\not\mid u$ and hence $1\notin F_t(u)$. This implies $\langle F_t(u)~|~u\in\hat{U}\rangle\subseteq \Delta_t(U)\setminus 1$. \\
For the reverse inclusion, let $G\in \Delta_t(U)\setminus 1$ be a facet. If $\dim G=\dim\Delta_t(U)$, then $G$ is also a facet of $\Delta_t(U)$ and hence $G=F_t(u)$ for some $u\in U$. As $1\notin G$, it follows that $x_1\not \mid u$ and hence $u\in \hat{U}$. It now suffices to show that $\Delta_t(U)\setminus 1$ is a pure $\dim\Delta_t(U)$-dimensional simplicial complex or, in other words, that there is no facet $G$ of $\Delta_t(U)\setminus 1$ with $\dim G<\dim\Delta_t(U)$. Assume on the contrary that such a facet $G$ exists. As $G\in \Delta_t(U)$, there exists a facet $F_t(u)\in \Delta_t(U)\setminus \left(\Delta_t(U)\setminus 1\right)$ such that $G\subset F_t(u)$. In particular, we must have $u\notin \hat{U}$. Let $\ell:=\max\{k~|~x_1^k|u\}$, i.e.,
$u=x_1^{\ell}\hat{u}$, where $x_1\not\mid \hat{u}$ and $\hat{u}\in U$, i.e., $\hat{u}\in \hat{U}$. By assumption, we have $\ell\geq 1$. We set 
$v=x_2^{\ell}\hat{u}$. As $U$ is shifted, we conclude that  $v\in \hat{U}$ and thus $F_t(v)\in \Delta_t(U)\setminus 1$. Hence, if we show that $G\subseteq F_t(v)$, we get a contradiction to the maximality of $G$ and the claim will follow. 
If $\ell\leq t$, we have 
$$
F_t(v)=\left(F_t(u)\setminus\{1\}\right)\cup\{2\ell+1\}
$$
and, since $G\subseteq F_t(u)\setminus \{1\}$, we conclude that $G\subseteq F_t(v)$ in this case. If $\ell>t$, then
$$
F_t(v)=\left(F_t(u)\setminus\{1\}\right)\cup\{\ell+t+1\}
$$
and the same argument as before shows $G\subseteq F_t(v)\in \Delta_t(U)\setminus 1$. 
\end{proof}

Combining the previous two lemmas the next statement is almost immediate.

\begin{theorem}
     \label{thm:vertexdecomposable}
Let $U\subset P(m)$ be a shifted order ideal and let $t \ge 0$ be any integer. Then $\Delta_t(U)$ is vertex decomposable with shedding vertex $1$. In particular, $\Delta_t(U)$ is shellable.
\end{theorem}

\begin{proof} 
By \cite{Mu-07} we may assume $0 \le t \le d_{\max} (U)$. 
We use double induction on $\#U$ and the number of variables. By our assumption on $U$, the variables $x_1,\ldots,x_m$ are in $U$, and so $\# U \ge m+1$. If $\# U = m+1$ we are done by Example~\ref{exa:m=1}(ii). This also takes are of the case where $m =0$. 

Let $\# U \ge m+2$.  If $m=1$, then $\Delta_t(U)\setminus 1=\langle F_t(1)\rangle$ (see Example~\ref{exa:m=1}(i)),  and thus it is just a single $\dim\Delta_t(U)$-dimensional simplex. As such it is vertex decomposable. If $m\geq 2$, then, by \Cref{lemma:deletion}, $\Delta_t(U)\setminus 1$ is isomorphic to the squeezed complex $\Delta_t(\hat{U})$, where $\hat{U}=\{u\in U~|~x_1\not\mid u\}$ is a shifted order ideal in variables $x_2,\ldots,x_m$. It follows from the induction hypothesis that $\Delta_t(U)\setminus 1$ is vertex decomposable. By \Cref{lemma:link} $\link_{\Delta_t(U)}(1)$ is isomorphic to $\Delta_0(\widetilde{U})$ and a cone over $\Delta_{t-1}(\widetilde{U})$ if $t=0$ and $t\geq 1$, respectively, where  $\widetilde{U}=\{u\in U~| ~x_1\cdot u  \in U\}$. As $\widetilde{U}$ is a shifted order ideal with $\#\widetilde{U}<\#U$, we can apply the induction hypothesis and as coning preserves vertex decomposability, we infer that $\link_{\Delta_t(U)}(1)$ is vertex decomposable in both cases. The claim follows.
\end{proof}

We can use \Cref{thm:vertexdecomposable} to compute the $h$-vector of $\Delta_t(U)$. 
\begin{corollary}\label{cor:hVector}
Let $U\subset P(m)$ be a shifted order ideal and $0\leq t\leq d_{\max}(U)$.  Then
$$h_i(\Delta_t(U))=\#\{u\in U~|~\deg(u)=i\}.$$
\end{corollary}

\begin{proof}
We show the statement by induction on $\#U$. If $\#U=m+1$, then Example~\ref{exa:m=1}(ii)) implies the claim. 

If $\#U\geq m+2$, then it follows from \Cref{thm:vertexdecomposable} and \eqref{h-vector:vertexdecomposable} that for $0\leq i\leq d$:
\begin{equation}\label{eq:h sum}
h_i(\Delta_t(U))=h_i(\Delta_t(U)\setminus 1)+h_{i-1}(\link_{\Delta_t(U)}(1)).
\end{equation}
Combining \Cref{lemma:deletion} and the induction hypothesis we obtain
\begin{equation}\label{eq:h deletion}
h_i(\Delta_t(U)\setminus 1)=\#\{u\in U~|~x_1\not \mid u,\;\deg(u)=i\}.
\end{equation}
Similarly, \Cref{lemma:link} allows us to use the induction hypothesis for the link and, using, that coning only appends $0$s to the $h$-vector, we conclude
\begin{equation}\label{eq:h link}
h_{i-1}(\link_{\Delta_t(U)}(1))=
\#\{u\in U~|~x_1\cdot u\in U,\deg(u)=i-1\}=\#\{u\in U~|~ x_1|u,\; \deg(u)=i\}.
\end{equation}
The claim follows from \eqref{eq:h sum}, \eqref{eq:h deletion} and \eqref{eq:h link}.
\end{proof}

\subsection{Squeezed complexes and stable operators}\label{subsect:squeezed and stable} 
We are now going to study the Stanley-Reisner ideals of the squeezed complexes defined in \Cref{subsect:Comb squeezed}. In particular, we will provide an explicit description of the minimal generators, generalizing results from \cite{Mu-07}. The crucial tool will be stable operators, whose definition we now recall. 

In the following, we let $P(\infty)$ be the polynomial ring $\K[x_1,x_2,\ldots]$ in infinitely many variables over the field $\K$  and we write $\Mon(\infty)$ for its set of monomials. 
 Given a map $\tau:\;\Mon(\infty)\to \Mon(\infty)$ we denote by $\tau(I)$ the monomial ideal generated by the set $\tau(G(I))=\{\tau(u)~|~u\in G(I)\}$. Note that one always has $G(\tau(I))\subseteq \tau(G(I))$. For a monomial $u \in P(\infty)$, we set $\max (u) := \max \{i ~|~ x_i \text{ divides } u\}$. If $M$ is a  graded module over a Noetherian polynomial ring $P(m)$ its \emph{graded Betti numbers} are the integers $\beta_{i, j} := \beta_{i, j} (M) = \dim_{\K} [\Tor_i^{P(m)} (M, \K)]_j$, where $i$ and $j$ are any integers. 
 
 For a finitely generated monomial ideal $I \subset P(\infty)$, we put
 \[
 \max(I) := \max \{ \max (u) ~|~ u \in G(I)\}. 
 \]
 Moreover, we write $\beta_{i, j} (I)$ for the graded Betti numbers of $I \cap P(\max(I))$ as a module over $P(\max(I))$. Note that the graded Betti numbers of $I \cap P(m)$ over $P(m)$ are the same for all $m \ge \max(I)$. 

\begin{definition}
A map $\tau:\Mon(\infty)\to \Mon(\infty)$ is called a \emph{stable operator} if the following conditions are satisfied:
\begin{asparaenum}

\item[(a)] If $I\subset P(\infty)$ is a finitely generated strongly stable ideal, then $\beta_{ij}(I)=\beta_{ij}(\tau(I))$ for all $i,j$.
\item[(b)] If $I,J\subset P(\infty)$ are strongly stable ideals with $I\subset J$, then $\tau(I) \subset \tau(J)$.
\end{asparaenum}
\end{definition}
A typical example of a stable operator is given by polarization, sending a monomial $x_{1}^{a_1}x_2^{a_2}\cdot x_{m}^{a_m}\in \Mon(\infty)$ to 
$$
x_{11}x_{12}\cdots x_{1a_1}\cdots x_{m1}\cdots x_{ma_m}\in \K[x_{ij}~|~i,j\in \N]\cong P(\infty).
$$ 
More generally, Murai \cite[Proposition 1.9]{Mu-07} showed the following.

\begin{proposition}
Let $\mathbf{a}=(a_1,a_2,a_3,\ldots)$ be a non-decreasing sequence of non-negative integers with $a_1=0$ 
and let 
$\tau_{\mathbf{a}}:\;\Mon(\infty)\to \Mon(\infty)$ defined by mapping a monomial $x_{i_1}x_{i_2}\cdots x_{i_s}$ with $i_1\leq i_2\leq \cdots \leq i_s$ to $x_{i_1}x_{i_2+a_2}\cdots x_{i_s+a_s}$.
Then $\tau_{\mathbf{a}}$ is a stable operator.
\end{proposition}

The next result from \cite{Mu-07} will be used in \Cref{sect: Alg Prop} to compute the graded Betti numbers of the squeezed complex $\Delta_t(U)$. In the sequel, we denote by $\gin (I)$ the generic initial ideal of a homogeneous ideal $I \subset P(m)$ with respect to the degree reverse lexicographic order. If $I \subset P(\infty)$ is a finitely generated monomial ideal, its generic initial ideal is $\gin (I) = \gin (I \cap P(\max(I))) P(\infty)$. Note that $\gin (I \cap P(m)) P(\infty) =  \gin (I \cap P(\max(I))) P(\infty)$ for every $m \ge \max(I)$. 

\begin{theorem}\cite[Theorem 1.6]{Mu-07}
     \label{Thm:stableOperator}
Let $\tau:\Mon(\infty)\to\Mon(\infty)$ be a stable operator. If $I\subset P(\infty)$ is a finitely generated strongly stable ideal, then $\gin(\tau(I))=I$.
\end{theorem}

In the following, we will be interested in a specific set of stable operators:

\begin{definition}
     \label{def:map phi}
For a non-negative integer $\ell\in \Z$, define the map $\phi_\ell:\Mon(\infty) \to \Mon (\infty)$ as follows: Given a monomial $u=x_{i_1}x_{i_2}\cdots x_{i_s} \in \Mon(\infty)$ $1 \le i_1 \le i_2 \le \cdots \le i_s$, set 
\begin{equation*}
\phi_{\ell} (u) =
\prod_{k = 1}^{\ell} x_{i_k+2(k-1)}\cdot \prod_{k=\ell+1}^s x_{i_{k} + k + \ell-1} 
\end{equation*} 
\end{definition}
Notice that $\phi_{\ell}$ maps every monomial onto a  monomial of the same degree. In fact, the map $\phi_{\ell}$ is nothing but the stable operator $\tau_{\mathbf{a}(\ell)}$, where $\mathbf{a}(\ell)=(a_1,a_2,a_3,\ldots)$ is given by
\[
a_k=\begin{cases}
2 (k-1),\quad \mbox{if } 1\leq k\leq \ell\\
k+\ell-1,\quad \mbox{if } k\geq \ell+1.
\end{cases}
\]
In particular, for $\ell=0$, the map $\phi_\ell$ is the stable operator corresponding to polarization. 

We record the previous observations in the following lemma:
\begin{lemma}
$\phi_{t}$ is a stable operator. 
Moreover, if $u\in \Mon(\infty)$, then the monomial $\phi_{t} (u)$ is squarefree.
\end{lemma}

The main result of this section shows that we can use the map $\phi_t$ to compute the Stanley-Reisner ideals of the squeezed complexe $\Delta_t(U)$. 
More precisely:

\begin{theorem} 
      \label{thm:SR-ideal}
Let $U\subset P(m)$ be a shifted order ideal. For any integer $t \ge 0$, one has 
$$
I_{\Delta_t(U)}=\phi_t(I(U)).
$$
\end{theorem}
The proof will follow along the same lines as the proof of Proposition 4.1 in \cite{Mu-07}, where the statement is shown if $t \ge d_{\max} (U)$.  We include it in full detail as a service to the reader. 

\begin{proof} 
Let $d:=d_{\max}(U)$ and let $0 \le t \le d_{\max} (U)$.
We first show that $\phi_t(U)\subseteq I_{\Delta_t(U)}$. 
Let $u=x_{i_1}\cdots x_{i_s}\in G(I(U))$ with $1\leq i_1\leq \cdots \leq i_s$. We write $f(u)$ for the support of $\phi_t(u)$, i.e.,
$$
f(u)=\{\ell~|~x_\ell|\phi_t(u)\}.
$$
We need to show that $f(u)\notin \Delta_t(U)$. For this it suffices to show that $f(u)\not\subseteq F_t(w)$ for any $w\in U$. Let $w=x_{j_1}\cdots x_{j_k}\in U$ with $j_1\leq \cdots \leq j_k$. We consider two cases, each having several subcases.\medskip

\noindent
{\sf Case 1:} $k<s$. 
There are three subcases: \smallskip 

\noindent
{\sf Case 1.1:} $k<s\leq t$.\\
We have $\phi_t(u)=\prod_{\ell=1}^sx_{i_\ell+2(\ell-1)}$ and hence 
\begin{equation}\label{eq:nonface}
f(u)=\{i_1,i_2+2,\ldots,i_s+2(s-1)\}.
\end{equation}
 Since $k<t$, it further holds that
\begin{equation}\label{eq:face}
F_t(w)=\{j_1,j_1+1\}\cup\cdots \cup\{j_k+2(k-1),j_k+2k-1\}\cup[m+2k+1,\ldots,m+t+d].
\end{equation}
 We write $F_t^{(\mathrm{in})}(w)$ and $F_t^{(\mathrm{fin})}(w)$ for the initial and final segment of $F_t(w)$, respectively, i.e.,
$$
F_t^{(\mathrm{in})}(w)=\{j_1,j_1+1\}\cup\cdots \cup\{j_k+2(k-1),j_k+2k-1\}
$$ 
and
$$
 F_t^{(\mathrm{fin})}(w)=[m+2k+1,\ldots,m+t+d].
$$
It follows from \eqref{eq:nonface} and the assumption $k<s$ that $\#(f(u)\cap F_t^{(\mathrm{in})}(w))\leq k$. Moreover, as $i_s+2(s-1)\leq m+2(s-1)$ and $\#[m+2k+1,m+2(s-1)]=2(s-k-1)$, we conclude from the form of $f(u)$ (see \eqref{eq:nonface}) that 
$$
\#(f(u)\cap  F_t^{(\mathrm{fin})}(w))\leq s-k-1.
$$
It follows that $\#(f(u)\cap F_t(w))\leq k+(s-k-1)=s-1$ and hence $f(u)\not \subseteq F_t(w)$, as desired.\smallskip 

\noindent
{\sf Case 1.2:} $k\leq t<  s$.\\
Since $s>t$,  we have $\phi_t(u)=\prod_{\ell=1}^tx_{i_\ell+2(\ell-1)}\cdot\prod_{\ell=t+1}^sx_{i_\ell+\ell+t-1}$ and thus 
\begin{equation}\label{eq:nonface:t<s}
f(u)=\{i_1,i_2+2,\ldots,i_t+2(t-1)\}\cup \{i_{t+1}+2t,i_{t+2}+2t+1\ldots,i_s+s+t-1\}.
\end{equation}
We set 
$$
f(u)^{(\mathrm{in})}=\{i_1,i_2+2,\ldots,i_t+2(t-1)\}\quad \mbox{and} \quad f(u)^{(\mathrm{fin})}=\{i_{t+1}+2t,i_{t+2}+2t+1,\ldots,i_s+s+t-1\}.
$$
Since $k<t$, the face $F_t(w)$ is of the form \eqref{eq:face}. As in Case 1.1 it follows that 
$$
\#(f(u)^{(\mathrm{in})}\cap F_t^{(\mathrm{in})}(w))\leq k \quad \mbox{and} \quad \#(f(u)^{(\mathrm{in})}\cap F_t^{(\mathrm{fin})}(w))\leq t-k-1.
$$
Using that $\#f(u)^{(\mathrm{fin})}=s-t$, we conclude $\#(f(u)\cap F_t(w))\leq k+(t-k-1)+s-t=s-1$, which implies $f(u)\not\subseteq F_t(w)$.\smallskip 

\noindent
{\sf Case 1.3:} $t\leq k<s$.\\
Since $s>t$, the support $f(u)$ is of the form \eqref{eq:nonface:t<s}. As $t\leq k$, the face $F_t(w)$ is given by
\begin{align}\label{eq:face:t<=k}
F_t(w)=&\{j_1,j_1+1\}\cup\cdots \cup\{j_t+2(t-1),j_t+2t-1\}\\
&\cup\{j_{t+1}+2t,j_{t+2}+2t+1,\ldots,j_k+k+t-1\}\cup[m+t+k+1,\ldots,m+t+d].\notag
\end{align}
We set 
\begin{align*}
F_t^{(\mathrm{in})} (w) & =\{j_1,j_1+1\}\cup\cdots \cup\{j_t+2(t-1),j_t+2t-1\},\\
F_t^{(\mathrm{mid})}(w) & =\{j_{t+1}+2t,j_{t+2}+2t+1,\ldots,j_k+k+t-1\} \quad\mbox{ and }\\
F_t^{(\mathrm{fin})}(w) &=[m+t+k+1,\ldots,m+t+d]
 \end{align*}
  for the initial, middle and final segment of $F_t(w)$. If $f(u)^{(\mathrm{in})}\not\subseteq F_t(w)$, we clearly have $f(u)\not\subseteq F_t(w)$ and the claim follows. So, assume $f(u)^{(\mathrm{in})}\subseteq F_t(w)$. We then must have $i_p+2(p-1)\geq j_p+2(p-1)$ for $1\leq p\leq t$, i.e., $i_p\geq j_p$ for $1\leq p\leq t$. In particular, $i_{t+1}+2t>i_t+2t-1\geq j_t+2t-1$, which implies $f(u)^{(\mathrm{fin})}\cap F_t^{(\mathrm{in})}(w)=\emptyset$. Moreover, we trivially have $\#(f(u)^{(\mathrm{fin})}\cap F_t^{(\mathrm{mid})}(w))\leq \#(F_t^{(\mathrm{mid})}(w))=k-t$. Since 
$$
i_\ell+\ell+t-1\leq m+\ell+t-1\leq m+t+k
$$
if $\ell\leq k+1$, we infer
$$
\#(f(u)^{\mathrm{fin}}\cap F_t^{(\mathrm{fin})}(w))\leq s-k-1.
$$
Summarizing the previous discussion, we obtain
$$
\#(f(u)\cap F_t(w))=\#(f(u)^{(\mathrm{in})})+\#(f(u)^{(\mathrm{fin})}\cap F_t(w))\leq t+0+(k-t)+(s-k-1)=s-1
$$
and thus $f(u)\not\subseteq F_t(w)$, as desired.\\
This finishes the proof of Case 1.\medskip

\noindent
{\sf Case 2:} $s\leq k$. 
We first consider the following subcase:\smallskip 

\noindent
{\sf Case 2.1:} $s\leq k$ and $s\leq t$.\\
Since $s\leq t$, the support $f(u)$ is of the form \eqref{eq:nonface}. As $s\leq t$, the face $F_t(w)$ contains the set 
$$
\{j_1,j_1+1\}\cup \cdots \cup\{j_s+2(s-1),j_s+2s-1\}
$$
and those are the smallest elements of $F_t(w)$. Assume, by contradiction, that $f(u)\subseteq F_t(w)$. It then follows that $i_\ell+2(\ell-1)\geq j_\ell+2(\ell-1)$, i.e., $i_\ell\geq j_\ell$ for all $1\leq \ell\leq s$. Since $U$ is shifted and $u\notin U$, it follows that $x_{j_1}\cdots x_{j_s}\notin U$. On the other hand, $x_{j_1}\cdots x_{j_s}$ divides $w$. Since $U$ is an order ideal and $w\in U$, it follows that $x_{j_1}\cdots x_{j_s}\in U$, a contradiction.\smallskip 

\noindent
{\sf Case 2.2:} $t<s\leq k$.\\
Since $t<s$ and $t<k$ the sets $f(u)$ and $F_t(w)$ are of the form \eqref{eq:nonface:t<s} and \eqref{eq:face:t<=k}, respectively. Assume that $f(u)\subseteq F_t(w)$. Similarly to the previous case, it follows that 
$i_\ell+2(\ell-1)\geq j_\ell+2(\ell-1)$, i.e., $i_\ell \geq j_\ell$ for $1\leq \ell\leq t$. Moreover, we must have $i_{\ell}+\ell+t-1\geq j_{\ell}+\ell+t-1$, i.e., again $i_\ell\geq j_\ell$ for $t+1\leq \ell \leq s$. Since $U$ is shifted and $u\notin U$, it follows that $x_{j_1}\cdots x_{j_s}\notin U$ Since $x_{j_1}\cdots x_{j_s}$ divides $w$, which lies in $U$, and $U$ is an order ideal, we conclude $x_{j_1}\cdots x_{j_s}\in U$, a contradiction.

It remains to show that $I_{\Delta_t(U)}\subseteq\phi_t(U)$. This will be a consequence of  the fact that $I_{\Delta_t(U)}$ and $\phi_t(U)$ have the same Hilbert function. Let $\delta:=\dim \Delta_t(U)$. It follows from \Cref{thm:vertexdecomposable} and \cite[Lemma 3.2]{Mu-07} that 
\[
\dim_{\mathbb{K}}\left(P(m)/(\mathrm{gin}(I_{\Delta_t(U)})\cap P(m))\right)_i= \begin{cases} 
h_i(\Delta_t(U)) & \text{ if } 0\leq i\leq \delta+1 \\
0 & \text{ if } i>\delta+1. 
\end{cases}
\]
By \Cref{cor:hVector}, we have
$$
h_i(\Delta_t(U))= 
\#\{u\in U~|~\deg(u)=i\}.
$$
The latter is also the Hilbert function of $P(m)/(I(U)\cap P(m))$ in degree $i$,  which implies that $I(U)\cap P(m)$ and $\mathrm{gin}(I_{\Delta_t(U)})\cap P(m)$ have the same Hilbert function. Since none of the minimal generators of neither $I(U)$ nor $\mathrm{gin}(I_{\Delta_t(U)})$ is divisible by $x_\ell$ for $\ell>m$ (see \cite[Lemma 3.2]{Mu-07}), it follows that $I(U)$ and $\mathrm{gin}(I_{\Delta_t(U)})$ have the same Hilbert function (considered in $P(m+\delta+1)$). As $\phi_t(\cdot)$ is a stable operator, \Cref{Thm:stableOperator} implies that $\mathrm{gin}(\phi_t(I(U)))=I(U)$. Since taking the generic initial ideal preserves the Hilbert function, we conclude that $\phi_t(I(U))$ and $I_{\Delta_t(U)}$ have the same Hilbert function.
\end{proof}

Recall that a simplicial complex $\Delta$ on vertex set $[n]$ is \emph{shifted} if for every face $F\in \Delta$ and $i\in F$, also $F\setminus\{i\}\cup\{j\}$ lies in $\Delta$ for every $i<j\notin F$. It is straightforward to see that the Stanley-Reisner ideal $I_\Delta$ of a shifted simplicial complex $\Delta$ is squarefree strongly stable, i.e., we may replace a variable dividing a monomial $u\in G(I_\Delta)$ with a variable of a smaller index to obtain another monomial in $G(I_\Delta)$ as long as the resulting monomial is squarefree.

It follows from the previous theorem that pure shifted complexes are squeezed complexes up to coning. 

\begin{corollary}
     \label{cor:shifted are squeezed} 
Let $\Delta$ be a $(d-1)$-dimensional pure shifted simplicial complex. Then $\Delta$ is the join of a  squeezed complex $\Delta_0 (U)$ and a $(d-1-d_{\max}(U))$-dimensional simplex (which is possible the empty set) for some finite shifted order ideal $U$. 
\end{corollary} 

The argument below shows explicitly how to construct the order ideal $U$ and to determine the dimension of the simplex in the join. 

\begin{proof}[Proof of Corollary~\ref{cor:shifted are squeezed}] 
Since the statement is combinatorial, it is harmless to assume that we work over a base field $\K$ of characteristic zero. Denote 
by $n$ the cardinality of the vertex of $\Delta$. Set $m = \max (\gin (I_{\Delta}))$ and 
$\delta = \max \{\deg u ~|~ u \in G(\gin (I_{\Delta})) \}$. Then Lemma~\ref{lem:shifted vs strongly stable} gives that 
$U = \Mon (m) \setminus (\gin (I_{\Delta}) \cap P(m))$ is a finite order ideal with $d_{\max} (U) = \delta-1$. Since $\Delta$ is shifted by assumption \cite[Theorem 1.3]{AHH} shows that its Stanley-Reisner ideal is $I_{\Delta} = \phi_0 (I(U) ) \cdot P(n)$. 

Since $\Delta$ is pure, it is Cohen-Macaulay \cite[Section 4]{Ka-02} and we infer that  $x_m^\delta$ is a minimal generator of $\gin (I_{\Delta})$ (see e.g., \cite[Lemma 3.2(ii)]{Mu-07}). Hence, we get 
$\delta + m \le n+1$ by \cite[Lemma 1.1]{AHH}. If follows that $m + d_{\max} (U) \le n$. Using that $I_{\Delta_0 (U)}=\phi_0 (I(U)) \subset P(m+ d_{\max} (U))$ by Theorem~\ref{thm:SR-ideal}, we conclude that $\Delta$ is isomorphic to the join of $\Delta_0 (U)$ and a simplex of dimension $n-1-m-d_{\max} (U)$. Since $d=n-m$ by \cite[Lemma 3.2(i)]{Mu-07}, it follows that  $n-1-m-d_{\max}(U)=d-1-d_{\max(U)}$. 
\end{proof}


\section{Algebraic Properties } 
     \label{sect: Alg Prop}

As applications of the above results we now discuss algebraic properties of the Stanley-Reisner ideals of $t$-squeezed complexes. The starting point is the following result: 

\begin{proposition}
    \label{prop:gin squeezed}
Let $U\subset P(m)$ be a shifted order ideal. If $t \ge 0$ is any integer then one has 
\[
\gin (I_{\Delta_t (U)} )= I(U) \cdot P(m + t + d_{\max}(U)).
\]
\end{proposition}

\begin{proof} 
Since $\phi_t$ is a stable operator this follows by combining Theorems \ref{thm:SR-ideal} and   \ref{Thm:stableOperator}. 
\end{proof}

In order to state a first consequence, recall that a homogeneous ideal $I$ is \emph{componentwise linear} if, for
all integers $j$, the ideal $I_{\langle j \rangle}$ has a $j$-linear resolution. Here $I_{\langle j \rangle}$ denotes the ideal 
generated by all homogeneous elements of degree $j$ in $I$. This important class of ideals was introduced in \cite{HH-Nagoya}.  It includes, for example, all stable monomial ideals and homogeneous ideals with linear quotients. Furthermore, ideals with extremal homological properties are often componentwise linear.

\begin{corollary}
    \label{cor:comp linear}
The Stanley-Reisner ideal of any squeezed complex $\Delta$ is componentwise linear. The Alexander dual of $\Delta$ is sequentially Cohen-Macaulay. 
\end{corollary}

\begin{proof}
By Proposition~\ref{prop:gin squeezed} and Lemma~\ref{lem:shifted vs strongly stable}, the ideal $\gin (I_{\Delta_t (U)} )$ is strongly stable. Moreover, using again that $\phi_t$ is a stable operator, 
Theorem~\ref{thm:SR-ideal} and Proposition~\ref{prop:gin squeezed} show that $\gin (I_{\Delta_t (U)} )$ and $I_{\Delta_t (U)}$ have the same graded Betti numbers. Hence the first claim follows by \cite{ArHeHi} in characteristic zero and by \cite[Theorem 2.3]{NR} in arbitrary characteristic. It implies the second assertion by \cite[Theorem 8.3.20]{HH}. 
\end{proof} 

Using the Eliahou-Kervaire formula, one can read off the graded Betti numbers of a squeezed  complex from the corresponding order ideal. \medskip

\begin{corollary}
   \label{cor:Betti numbers} 
The graded Betti numbers of an ideal of a squeezed complex $\Delta_t (U)$ are 
\[
\beta_{i, i + j} ( I_{\Delta_t (U)} ) = \sum_{u \in G(I(U)), \ \deg u = j} \binom{\max (u) - 1}{i}, 
\] 
where $\max (u)$ is the largest index of a variable dividing the monomial $u$. 
\end{corollary} 

\begin{proof}
Since  $I_{\Delta_t (U)}$ and $\gin (I_{\Delta_t (U)} )$ have the same graded Betti numbers this follows by Proposition~\ref{prop:gin squeezed} and the Eliahou-Kervaire formula (see \cite{EK}) applied to $I(U)$. 
\end{proof}

We now turn our attention to the so-called Lefschetz properties. Recall that a standard graded 
Cohen-Macaulay $\K$-algebra of (Krull) dimension $d$ has the \emph{weak Lefschetz property (WLP)} if 
there is a linear system of parameters $\ell_1,\ldots,\ell_d \in A$ and a linear form $\ell \in A$ such 
that the Artinian reduction $\overline{A} = A/(\ell_1,\ldots,\ell_d) A$ of $A$ satisfies 
that the multiplication map $\times \ell \colon [\overline{A}]_i \to [\overline{A}]_{i+1}$ has maximal rank for every integer $i$. One says that $A$ has the \emph{strong Lefschetz property (SLP)} if every multiplication map $\times \ell^j \colon [\overline{A}]_i \to [\overline{A}]_{i+j}$ has maximal rank for every $i$. If one of these conditions holds, the linear form $\ell$ is called a \emph{weak Lefschetz element} (resp.\ \emph{strong Lefschetz element}) of $\overline{A}$. 

There is a rich literature on the problem of deciding whether a Cohen-Macaulay algebra has one of the Lefschetz properties. It has been studied from many
 different points of view, applying tools from representation theory,
 topology, vector bundle theory, lozenge tilings, splines, hyperplane arrangements,  inverse systems, 
 differential geometry, among others (see, e.g., \cite{BMMNZ, BK, CN, CHMN, DIV, HSS, HMNW, HMMNWW, KRV, MMO, NT, St-faces}).

 For example, let $\Delta$ be the boundary complex of a simplicial polytope $P$ and let $A$ be the Stanley-Reisner ring of $\Delta$. Stanley showed that $A$ has the SLP and deduced from this the necessity of the conditions on the $h$-vector of $P$ in the so-called $g$-Theorem that characterizes $h$-vectors of simplicial polytopes. 

From now on we assume that the base field $\K$ is infinite. Then it is well-known that if $A$ is  a $d$-dimensional Cohen-Algebra with the WLP (resp.\ SLP) and $\ell_1,\ldots,\ell_d \in A$ are general linear elements, then a general linear element $\ell$ in the Artinian reduction $\overline{A} = A/(\ell_1,\ldots,\ell_d) A$ is a weak (resp.\ strong) strong Lefschetz element of $\overline{A}$. 

For a simplicial complex $\Delta$, we say that $\Delta$ has the WLP or SLP if its Stanley-Reisner ring $A = \K[\Delta]$ has the corresponding property. We are going to show that the Lefschetz properties of a squeezed complex can be read off from its order ideal. This requires some preparation. 

\begin{lemma}
    \label{lem:reduce Lefschetz to order ideal} 
Let $U\subset P(m)$ be a shifted order ideal. Fix any integer $t \ge 0$,  
let $A= \K[\Delta_t (U)]$ be the  Stanley-Reisner ring of $\Delta_t (U)$,  and let 
$\overline{A} = A/(\ell_1,\ldots,\ell_d) A$ be its Artinian reduction, where $\ell_1,\ldots,\ell_d$ are 
general  linear elements of $A$. 
For a general linear form $\ell \in \overline{A}$ and integers $i, j$ with $j \ge 1$, consider the multiplication 
map $\phi = \times \ell^j \colon [\overline{A}]_i \to [\overline{A}]_{i+j}$ and the map 
\[
\psi  \colon  [U]_i \cup \{0\} \to [U]_{i+j} \cup \{0\}, \quad u \mapsto \begin{cases}
u x_m^j, & \text{ if } u x_m^j  \in U \\
0, & \text{ otherwise}. 
\end{cases}
\]
Then $\phi$ has maximal rank if and only if $\psi$ does. 
\end{lemma}

\begin{proof}
It follows as in \cite{W} that $[\overline{A}]_i \stackrel{\ell^j}{\longrightarrow} [\overline{A}]_{i+j}$ has maximal rank if and only if the multiplication map $[P(m)/I(U)]_i \stackrel{x_m^j} {\longrightarrow}[P(m)/I(U)]_{i+j}$ does. Using that $[U]_i$ is a $K$-basis of $[P(m)/I(U)]_i$ by Proposition~\ref{prop:gin squeezed}, this implies the claim. 
\end{proof}

In order to discuss maximal rank of the above map $\psi$ it is convenient to introduce a concept modeled after the idea of a Borel generator of a strongly stable ideal (see e.g., \cite{FMS}). \medskip

\begin{definition}
    \label{def:shift gen}
(i) A monomial  $x_j u \in P(m)$ is said to be obtained by a \emph{shift move} from a monomial $x_i u$ if $j \ge i$. A monomial $u \in P(m)$ is called \emph{shift-generated} by a monomial $v \in P(m)$ if $u$ can be obtained from $v$ by a finite sequence of shift moves.     
    
(ii) Let $U \subset P(m)$ be a shifted order ideal. A set $G_s$  is a called a set of \emph{shift generators} of $U$ if the set of monomials that are shift-generated by some monomial in $G_s$ is  $U$. 
We denote by $G_s (U)$ the smallest set of shift generators of $U$. 
\end{definition}

\begin{example}
    \label{exa:shift gens}
There are exactly two shifted order ideals of $P(3)$ with $h$-vector $(1, 3, 3)$, namely 
\begin{align*}
U' & = \{1, x_3, x_2, x_1, x_3^2, x_2 x_3, x_2^2\}  \quad \text{ and } \\
U &= \{1, x_3, x_2, x_1, x_3^2, x_2 x_3, x_1 x_3\}.
\end{align*}
Their minimal sets of shift generators are $G_s (U') = \{1, x_1, x_2^2\}$ and $G_s (U) = \{1, x_1, x_1 x_3\}$.  
\end{example}

\begin{lemma}
    \label{lem:max rank of psi}
Let $U\subset P(m)$ be a shifted order ideal, and consider the map $\psi  \colon  [U]_i \cup \{0\} \to [U]_{i+j} \cup \{0\}$ introduced in Lemma~\ref{lem:reduce Lefschetz to order ideal}. Then one has: 
\begin{itemize}
\item[(i)] The map $\psi$ is surjective if and only if every monomial in $[U]_{i+j}$ is divisible by $x_m^j$. 

\item[(ii)] The map $\psi$ is injective if and only if $x_m^j [G_s (U)]_i \subset [U]_{i+j}$. 
\end{itemize}

\end{lemma}

\begin{proof}
Since the forward implications are straightforward we discuss only the backward implications. If 
$u \in [U]_{i+j}$ is divisible by $x_m^j$, then $u/x_m^j \in [U]_i$ as $U$ is an order ideal. 

For (ii), notice that $\psi$ is injective if and only if $x_m^j [U]_i \subset [U]_{i+j}$. Furthermore, if a monomial $v$ is shift-generated by some $u \in [G_s (U)]_i$ then $x_m^j v$ is shift-generated by $x_m^j u$. Combining these two facts, the claim follows.  
\end{proof}

The above discussion has the following consequences for the Lefschetz properties: 

\begin{proposition}
    \label{prop:char WLP}
Let $\Delta_t (U)$ be a squeezed complex with $h$-vector $(h_0,h_1,\ldots,h_d)$, where $d = d_{\max} (U)$. Let $\alpha$ be the least degree in which the $h$-vector attains its minimum, that is, 
\[
\alpha = \min \{i  ~|~  h_i = \max \{ h_k ~|~ 0 \le k \le d \} \}. 
\]
Then $\K[\Delta_t (U)]$ has the WLP if and only if the following conditions are satisfied:
\begin{itemize}
\item[(i)] Every $u \in [U]_j$ is divisible by $x_m$ for $j>\alpha$

\item[(ii)] $x_m [G_s (U)]_{j-1} \subset [U]_j$ for every $0<j\leq \alpha$. 
\end{itemize}

\end{proposition}

\begin{proof}
Denote by $\overline{A}$ a general Artinian reduction of $\K[\Delta_t (U)]$. By the definition of $\alpha$, the complex $\Delta_t (U)$ has the WLP if and only if the multiplication map 
$[\overline{A}]_{j-1} \stackrel{\ell}{\longrightarrow} [\overline{A}]_{j}$ is injective whenever $j \le \alpha$ and surjective whenever $j > \alpha$. The latter is equivalent to Conditions (i) and (ii) by Lemmas~\ref{lem:reduce Lefschetz to order ideal} and \ref{lem:max rank of psi}. 
\end{proof}

Similarly, one obtains the following result. 

\begin{proposition}
     \label{prop:char SLP}
Let $\Delta_t (U)$ be a squeezed complex with $h$-vector $h=(h_0,h_1,\ldots,h_d)$, where $d = d_{\max} (U)$. Then $\K[\Delta_t (U)]$ has the SLP if and only if, for any two integers $i, j$ with $0 \le i < j \le d$, one has
\begin{itemize}
\item[(i)] If $h_i \le h_j$, then $x_m^{j-i} [G_s (U)]_i \subset [U]_{j}$.

\item[(ii)] If $h_i \ge h_j$, then every $u \in [U]_j$ is divisible by $x_m^{j-i}$. 
\end{itemize}

\end{proposition}

We illustrate the above statements by an example. 

\begin{example}
     \label{exa:WLP SLP}
Let $U'$ and $U$ be the two shifted ordered ideals with $h$-vector $(1, 3, 3)$ as considered in Example~\ref{exa:shift gens}. Fix any integer $t \ge 0$. 
Proposition~\ref{prop:char WLP} shows that $\Delta_t (U')$ fails to have the WLP as $\alpha = 1$ and $x_2^2 \in [G_s (U')]_2$ is not divisible by $x_3$. However, the squeezed complex $\Delta_t (U)$ does have the SLP by Proposition~\ref{prop:char SLP}.
\end{example}


\section{Singularity Indices} 
   \label{sec:sing index}

Given any shifted order ideal $U \subset P(m)$, we introduced  squeezed complexes $\Delta_t (U)$, where $t$ is any integer with $0 \le t \le d_{\max} (U)$. The aim of this section is to characterize how close these simplicial complexes are to being pseudomanifolds. For $m \ge 2$, we will see that  $\Delta_0 (U)$ is never a pseudomanifold (see Proposition~\ref{prop:t=0} below), whereas $\Delta_{d_{\max}} (U)$ always is a pseudomanifold. Indeed, $\Delta_{d_{\max}}$ is even a ball \cite[Corollary 3.2]{Ka}. Intuitively, fixing the order ideal $U$, the complex $\Delta_t (U)$ is the closer to being a pseudomanifold the larger $t$ is. This section is an attempt to make this intuition more precise. To this end, we introduce an invariant called singularity index that measures how close a simplicial complex is to being a pseudomanifold. 

\begin{definition}
      \label{def:degree of a ridge}
The \emph{degree} of a ridge $F$ of a pure simplicial complex $\Delta$ is the number of facets of $\Delta$ containing $F$. It is denoted $\deg (F)$. 
\end{definition}

Thus, among the simplicial complexes with a fixed number of ridges and without boundary (i.e., the degree of every ridge is at least two), closed pseudomanifolds minimize the total degree, which is the sum of the degrees of all ridges. Similarly, if one considers all simplicial complexes with a given number of ridges in the interior, then open pseudomanifolds have the smallest total degree. To compare complexes with distinct numbers of ridges we use a relative version. 

\begin{definition}
      \label{def:sing index}
The \emph{singularity index} of a pure $d$-dimensional simplicial complex $\Delta$ is
\[
D (\Delta) = \frac{{\displaystyle \sum_{F \in \Delta,\ \dim F = d-1} \max \{0, \deg (F) - 2\}}}{f_{d-1} (\Delta)}. 
\]
\end{definition}

Hence a shellable simplicial complex $\Delta$ is a pseudomanifold if and only if $D (\Delta) = 0$ if and only if $\Delta$ is a shellable ball or sphere \cite{DK}. Moreover, if $\Delta$ is a pseudomanifold, then $D(\Delta)=0$. However, there are simplicial complexes that are not pseudomanifolds with singularity index $0$. Take, for example, a collection of $d$ simplices that is glued together along faces of codimension $\geq 2$. 

We want to remark, that there is a natural way to extend the notion of a singularity index to non-pure simplicial complexes. Instead of summing over ridges one could consider all those faces $F$ that lie in a facet $G$ with $\dim G=\dim F+1$. 

The singularity index is bounded. 

\begin{remark}
    \label{rem:trivial bounds}
If $\Delta$ is any simplicial complex then the degree of every ridge is at most $f_0 (\Delta) - \dim \Delta$, and so  $$0 \le D (\Delta) \le \max \{0, f_0 (\Delta) - \dim \Delta - 2 \}.$$ 

For any shifted order ideal $U \subset P(m)$, this gives that the degree of any ridge of a squeezed complex $\Delta_ t (U)$  is at most $m+1$ and $0 \le D(\Delta_t (U)) \le m-1$. Moreover, Theorem~\ref{thm:max order manifold} and Proposition~\ref{prop:t=0} below show that these bounds are sharp.
\end{remark}

Given a shifted order ideal $U$, the squeezed complex $\Delta_t(U)$ is always a pseudomanifold \cite{Ka} if $t\geq d_{\max}(U)$ and hence $D (\Delta_{t} (U)) = 0$ in this case.
It is natural to ask, if there are classes of shifted order ideals such that $D(\Delta_t(U))=0$ even if $t<d_{\max}(U)$. 

\begin{proposition}\label{prop: common divisor}
Let $U\subset P(m)$ be a shifted order ideal such that $x_m$ divides every monomial $u\in U$ of maximal degree. Then, $\Delta_{d_{\max}(U)-1}(U)$ is a ball, and hence
$D(\Delta_{d_{\max}(U)-1}(U))=0$.
\end{proposition}

\begin{proof}
Put $d=d_{\max}(U)$. We are going to show that $\Delta_{d}(U)$ is isomorphic to a cone over $\Delta_{d-1}(U)$. Let $u=x_{i_1}\cdots x_{i_s}\in U$. If $\deg u<d$, then we have $\{m+2d-1,m+2d-1\}\subseteq F_d(u)$. If $\deg u=d$, then $i_d=m$ by assumption and hence $i_d+2d-1=m+2d-1\in F_d(u)$. It follows that $\Delta_d(U)$ is a cone over the vertex $m+2d-1$, i.e., $\Delta_d(U)=\{m+d-1\}\ast \Gamma$ for some pure simplicial complex $\Gamma$. It suffices to show that $\Gamma$ is isomorphic to $\Delta_{d-1}(U)$. To this end, consider the map $\phi: [m+2d-1]\to [m+2d-2]\cup\{m+2d\}$ that is the identity on $[m+2d-2]$ and that maps $m+2d-1$ to $m+2d$. It is easily seen that $\phi$ induces the required isomorphism between $\Delta_{d-1}(U)$ and $\Gamma$. The claim now follows since $\Delta_d(U)$ is a ball \cite{Ka}. 
\end{proof}

 The goal of this section is to study the sequence $(D (\Delta_{t} (U)))_{t \ge 0}$. Since $D (\Delta_{t} (U)) = 0$ if $t \ge d_{\max} (U)$, we are interested in the initial part of this sequence. 

Before turning to the general case we will focus on complexes obtained from maximal order ideals. A \emph{maximal order ideal} of $P(m)$ is a subset $U \subset P(m)$ such that, for some integer $d \ge 1$, $U$  consists of \emph{all} monomials in $P(m)$ whose degree is at most $d$. 
Note that every order ideal $U$ of $P(1)$ with a least two elements is maximal and that $\Delta_t (U)$ is a pseudomanifold whenever $t\geq 0$ by Example~\ref{exa:m=1}. Thus, we assume $m \ge 2$ from now on. 

Apart from the cases mentioned above, squeezed complexes to maximal order ideals are never manifolds. 

\begin{theorem}
     \label{thm:max order manifold}
Let $m\geq2$ and $U \subset P(m)$ be a maximal order ideal with $d_{\max}(U)\geq 2$. Then the following conditions are equivalent: 
\begin{itemize}
\item[(a)] $\Delta_ t (U)$ is a ball.
\item[(b)] $\Delta_t(U)$ is a pseudomanifold.
\item[(c)] $t \ge d_{\max (U)}$. 
\end{itemize}
More precisely, if $0 \le t < d_{\max (U)}$, then $\Delta_t (U)$ has a ridge of (maximum) degree $m+1$. 
\end{theorem}

\begin{proof} 
Put $d = d_{\max (U)}$. 
Since by \cite[Corollary 3.2]{Ka}, $\Delta_t(U)$ is a shellable ball if $t\geq d$, it is enough to show that $[d+t-1]$ is a ridge of $\Delta_t (U)$ with degree $m+1$ whenever  
$0 \le t \le d$. Indeed, if $d\geq 2$, one computes that $F_t (x_1^{d-1}) = [d+ t - 1] \cup \{m+d+t\}$ and, for 
$j = 1,\ldots,m$, 
\[
F_t (x_1^{d-1} x_j) = [d+t-1] \cup \{d+t+j-1\}.
\]

\end{proof}

If $U$ is not a maximal order ideal, then $\Delta_t (U)$ can be a pseudomanifold even if $t < d_{\max (U)}$. 

\begin{example}
     \label{exa:almost max oder ideal} 
Consider the maximal order ideal $U \subset P(2)$ with $d_{\max} (U) = 3$. Put  
$U' = U \setminus \{x_1^3\}$. Then $\Delta_2 (U')$ and $\Delta_3 (U')$ are pseudomanifolds, and one gets for the sequence of singularity indices 
\[
D (\Delta_0 (U')) = \frac{7}{9} > \frac{4}{19} = D (\Delta_1 (U')) > 0 = D (\Delta_2 (U')) = D (\Delta_3 (U)). 
\]
\end{example}

However, a squeezed complex $\Delta_0 (U)$ is never a pseudomanifold. 

\begin{proposition}
     \label{prop:t=0}
Consider a maximal order ideal $U \subset P(m)$. Setting $d = d_{\max} (U)$, the complex $\Delta_0 (U)$ has $f$-vector 
\[
f = \left(1, m+d, \binom{m+d}{2},\ldots,\binom{m+d}{d}\right). 
\]
Its singularity index is $D(\Delta_0 (U)) = m-1$. 
\end{proposition}

\begin{proof}
By definition, $\Delta_0 (U)$ is a $(d-1)$-dimensional complex on $[m+d]$ with $|U| = \binom{m+d}{d}$ facets. Thus, $\Delta_0 (U)$ is the $(d-1)$-skeleton of an $(m+d-1)$-simplex, which implies the assertions. 
\end{proof}

\begin{remark}
It follows from \Cref{prop:t=0} that a squeezed complex $\Delta_t(U)$ can never be a sphere. Indeed, in this case, we would have $h_{d_{\max}(U)+t}(\Delta_t(U))=1$, which, by \Cref{cor:hVector}, can only happen if $t=0$. However, in this situation, $\Delta_t(U)$ is never a pseudomanifold.
\end{remark}
We now determine the sequence $(D (\Delta_{t} (U)))_{t \ge 0}$ in some cases. 

\begin{proposition}
    \label{prop:d=2}
Fix $m \ge 2$, and let $U \subset P(m)$ be the shifted order ideal consisting of all monomials whose degree is at most 2. Then the singularity indices of the squeezed complexes of $U$ are 
\[
D (\Delta_0 (U)) = m-1 > \frac{\binom{m+1}{2} -1}{\binom{m+3}{2}} = D (\Delta_1 (U)) > 0 = D (\Delta_2 (U)). 
\]
\end{proposition} 

\begin{proof}
By Proposition~\ref{prop:t=0} and \cite{Ka}, it suffices to determine $D (\Delta_1 (U))$. The facets of $\Delta_1 (U)$ are 
\begin{align*}
F_1 (1) & = \{m+1, m+2, m+3\}, \\
F_1 (x_i) & = \{i, i+1, m+3\}, \quad 1 \le i \le m, \\
F_1 (x_i x_j) & = \{i, i+1, j+2\},  \quad 1 \le i \le j \le m. 
\end{align*}
The ridges of $\Delta_1 (U)$ are precisely all 2-subsets of $[m+3]$, that is, $f_2 = f_2 (\Delta_1 (U)) = \binom{m+3}{2}$.  Most of the ridges are contained in at  most two facets. The exceptions are $\{1, 2\}$, which is contained in $m+1$ facets (see the proof of 
 Theorem~\ref{thm:max order manifold}) and $\{k, k+1\}$ with $2 \le k \le m$. The latter is contained in the $m-k+3$ facets $\{k-1, k, k+1\}$ and $\{k, k+1, \ell\}$, where $k+2 \le \ell, m+3\}$. Thus, we obtain 
\[
 f_2 \cdot D (\Delta_1 (U)) = m-1 + \sum_{k=2}^{m} (m-k+1) = \binom{m+1}{2} - 1. 
 \]
 Now the claim follows. 
\end{proof}

\begin{proposition}
    \label{prop:d=3}
Fix $m \ge 2$, and let $U \subset P(m)$ be the shifted order ideal consisting of all monomials whose degree is at most 3. Then the singularity indices of the squeezed complexes of $U$ are 
\begin{align*}
D (\Delta_0 (U)) = m-1&  > \frac{2 \binom{m+2}{3} + \binom{m}{2} -2}{\binom{m+4}{3}} = D (\Delta_1 (U)) \\[1ex]
& > \frac{\binom{m+2}{3} -1}{\binom{m+4}{3} + \binom{m+3}{3}} = D (\Delta_2 (U)) 
 >  0 = D (\Delta_3 (U)). 
\end{align*}
\end{proposition} 

\begin{proof} 
First, consider $\Delta_1 (U)$. We list only the ridges whose degrees are at least three. These are $\{a, a+1,b\}$ with $2 \le a \le m$ and $a+2 \le b \le m+4$, which has degree $m-a+3$, and $\{1,2,b\}$ with $3 \le b \le m+4$, which has degree $m+1$. Since the $h$-vector of $\Delta_1(U)$ can be obtained from the one of $\Delta_0(U)$ by appending a $0$, the number of ridges of $\Delta_1(U)$ is given by
 $$f_{3}(\Delta_1(U))=f_2(\Delta_0(U))+f_3(\Delta_0(U))=\binom{m+3}{2}+\binom{m+3}{3}=\binom{m+4}{3},$$ 
 where the second equality follows from \Cref{prop:d=2}. Also note that $f_4(\Delta_1(U))=f_3(\Delta_0(U))=\binom{m+3}{3}$. 

Second, we describe the ridges of $\Delta_2 (U)$ whose degrees are at least three. These are $\{a, a+1,b,b+1\}$ with $1 \le a < m$ and $a+3 \le b \le m+2$, which has degree $m-c+5$, and $\{a, a+1, a+2, a+3\}$ with $1 \le a \le m$. The latter has degree $m-a+3$ if $a \ge 2$, whereas $\{1,2,3,4\}$ has degree $m+1$. The same reasoning as for $\Delta_1(U)$ shows that the number of ridges of $\Delta_2(U)$ is 
$f_{4}(\Delta)= \binom{m+4}{3}+\binom{m+3}{3}$. 

Now straightforward computations give the formulas for the singularity indices of $\Delta_1 (U)$ and $\Delta_2 (U)$ as well as the claimed inequalities. 
\end{proof}

Observe that in  all the examples above the sequence $(D (\Delta_{t} (U)))_{t \ge 0}$ is weakly decreasing. In fact, this is always true. In order to establish this fact, the following result will be useful. 

\begin{lemma}
     \label{lem:technical}
Let $U \subset P(m)$ be a shifted order ideal. Consider any integer $t$ with $1 \le t \le d = d_{\max} (U)$.  Then on has: 
\begin{itemize}
\item[(a)] If $u \in U$ and $F_t (U) = \{k_1,\ldots,k_{d+t}\}$ with $k_1 < k_2 < \cdots < k_{d+t}$ then 
\[
F_{t-1} (u) = \{k_1,\ldots, k_{2t -1}, k_{2t+1} -1,\ldots,k_{d+t}-1\}, 
\]
that is,  the $(2t)$\textsuperscript{th} smallest element is omitted and all larger elements of $F_t (u)$ are reduced by one. 

\item[(b)] Let $u = x_{i_1} \cdots x_{i_s}$ and $v = x_{j_1} \cdots x_{j_r}$ be two distinct monomials of positive degrees in $U$ with $i_1 \le \cdots \le i_s$ and  $j_1 \le \cdots \le j_r$. Assume that $F_t (u) \cap F_t (v)$ is a ridge of $\Delta_t (U)$ and there is some integer $a \in [t]$ such that 
$i_{\ell} = j_{\ell}$ if $\ell < a$. Then the following statements are true: 
\begin{enumerate}
\item If $0<i_a < j_a$ or $j_a=0$, then $\deg u \ge \deg v$. 

\item If for some $b \in [t]$ with $b \ge a$ one has $i_a = i_{a+1} = \cdots = i_b$ and 
$j_a = \cdots j_b = i_a + 1$ and either $i_{a} < i_{b+1}$ or $i_{b+1}=0$, then $x_{i_a + 1} \cdot u = x_{i_a} \cdot v$. 

\end{enumerate}
 
\item[(c)] If $u, v \in U$ are two distinct monomials whose degrees are at most $t$ and such that $F_t (u) \cap F_t (v)$ is a ridge of $\Delta_t (U)$, then this is a ridge of degree two. 

\item[(d)] Let $u,v \in U$ be two distinct monomials and let $a\in \mathbb{N}$ such that the smallest $a$ elements in $F_t(u)$ and $F_t(v)$ coincide but this ist not true for the smallest $a+1$ elements in each set. If $a\leq 2t$, then $a$ is even.
\end{itemize}

\end{lemma}

\begin{proof}
(a) Considering the two cases where $\deg u \le t$ and $\deg u \ge t$ separately, this follows by straightforward computations. 

(b)(i) The statement is trivially true if $j_a=0$. So assume that $j_a>i_a> 0$. If $\deg u < \deg v$, then $m+ \deg v + \min \{t, \deg v \}$ is in $F_t (u)$, but not in $F_t (v)$. Since $i_a < j_a$ one also has that $i_a + 2a-2$ is in $F_t (u)$, but not in $F_t (v)$. It follows  that $F_t (u) \cap F_t (v)$ cannot be a ridge; a contradiction.  

(b)(ii) First assume that $0<i_a<i_{b+1}$. Using again that $i_a + 2a -2 \notin F_t(v)$ as well as $i_a + 2b = i_b + 2b = j_b + 2b -1 \notin F_t (u)$ as $i_{b+1} > i_b$ and $F_t(u)\cap F_t(v)$ is a ridge by assumption, we get
\[
|F_t (u) \cap F_t (v) \cap [i_a + 2b] | = |F_t (u)  \cap [i_a + 2b] | -1  = |F_t (v) \cap [i_a + 2b] | - 1. 
\]
Since $F_t (u) \cap F_t (v)$ is a ridge this implies $F_t (u) \cap [i_a + 2b +1, m+t+d] = F_t (v) \cap [i_a + 2b +1, m+t+d]$, which in turn gives $\deg u = \deg v$ and $i_{\ell} = j_{\ell}$ whenever $b < \ell \le s = r$. Now the claim follows. 
Finally assume $i_{b+1}=0 $. Since $m+2b+1$ is the $(2b+1)$\textsuperscript{st} smallest element of $F_t(u)$ and $j_b+2b-1\leq m+2b-1<m+2b+1$ and we have $j_b+2b-1\notin F_t(u)$. Now the same reasoning as above yields the claim.

(c) Assume on the contrary that there is another facet, $F_t (w)$, that contains the ridge $R = F_t (u) \cap F_t (v)$. We consider two cases. 

\emph{Case 1.} Assume $\deg w \le t$. Since $[m+2t+1, m+d + t] \subset F_t (u) \cap F_t (v) \cap F_t (w)$, this implies that $R \cap [m+2t]$ is a ridge of $\Delta_t (U_{\le t})$ whose degree is at least three. This is a contradiction because $\Delta_t (U_{\le t})$ is a pseudomanifold by \cite{Ka}.

\emph{Case 2.} Assume $\deg w > t$. Then $m+ t + \deg w$ is not in $F_t (w)$, but it is in $R$ as the degrees of both $u$ and $v$ are not greater than $t$. This is a contradiction to  $R \subset F_t (w)$. 

(d) It follows from the definition of $F_t(\cdot)$, that if $\ell$ is the smallest $(2k-1)$\textsuperscript{st} element of $F_t(u)$ for $1\leq k\leq t$, then $\ell+1\in F_t(u)$. This implies the claim.
\end{proof}

We are ready for the main result of this section. It establishes the announced monotonicity of the singularity index.  In that sense it confirms the intuition that, for a fixed order ideal $U$, the complex $\Delta_t (U)$ is less squeezed and the closer to being a pseudomanifold the larger $t$ is.

\begin{theorem} 
      \label{thm:sing index}
Let $U$ be a shifted order ideal. Then the sequence of non-negative fractions $D (\Delta_0 (U)), D (\Delta_1 (U)),\ldots,D (\Delta_{d_{\max} (U)} (U))$ is strictly decreasing until it reaches zero, that is, if $0 \le t < d_{\max (U)}$, then 
\[
D (\Delta_{t+1} (U)) \ \begin{cases}
< D (\Delta_t (U)) & \text{ if } D (\Delta_t (U)) \neq 0 \\
= 0 & \text { if } D (\Delta_t (U)) = 0.  
\end{cases}
\]
Moreover, $D(\Delta_{d_{\max}}(U))=0$.
\end{theorem}

\begin{proof}
We have already seen that $D(\Delta_{d_{\max}}(U))=0$. So assume that $1\leq t< d_{\max}(U)$.   
By Example~\ref{exa:m=1} we may assume $m \ge 2$. 
Set $d = d_{\max} (U)$ and consider any ridge 
\[
F_t (u_1) \cap \cdots \cap F_t (u_{\delta}) = \{q_1,\ldots,q_{d+t-1}\}  
\]
of $\Delta_t (U)$ with degree $\delta \ge 3$ where $q_1 < q_2 < \cdots < q_{d+t-1}$. Put $W = \{u_1,\ldots,u_{\delta}\}$. 

Our first goal is to show that $q_{2 t} = q_{2t -1} + 1$. 
We distinguish two cases.

\smallskip
{\sf Case 1.} Assume that the smallest $2 t$ elements in every set $F_t (u_i), \ i = 1,\dots, \delta$, are the same. 

As $\delta\geq 3$ Lemma~\ref{lem:technical}(c) implies that at most one of the monomials in $W$ has degree less than or equal to $t$. In particular, there exists a monomial $u \in W$ of degree larger than $t$. Then the smallest $2t$ elements in $F_t (u)$ are unions of pairs of consecutive integers. By assumption, these elements are  $\{q_1,\ldots,q_{2t}\}$. It follows that    $q_{2 t} = q_{2t -1} + 1$, as desired. 
\smallskip 

{\sf Case 2.}  Assume that the smallest $2 t$ elements in the sets $F_t (u_i), \ i = 1,\dots, \delta$, are not the same. 

In this case we proceed in several steps.
\smallskip

(I) We claim that $1$ is not in $W$. Indeed, since $F_t (1) = [m+1, m+t+d]$ and $t\geq 1$ the set $F_t (1) \cap F_t (u)$ with $u \in U$ is a 
ridge of $\Delta_t (U)$ if and only if $u = x_m^s$ for some $s \in [d]$. Combined with the fact that 
$F_t (x_m^s) = [m, m+ t + d] \setminus \{ m+s + \min \{s, t\} \}$, it follows that $F_t (1) \cap F_t (u)$ is a ridge of degree 
$2 < \delta$. This  proves the  claim. Hence, every element in $W$ has positive degree. 
\smallskip

(II)  We are now going to show that every element in $W$ has degree greater than or equal to $a$. 
By the assumption of Case 2 and \Cref{lem:technical} there is some integer $a$ such that the smallest $2a-2$ elements in every set $F_t (u_i), \ i = 1,\dots, \delta$, are the same, but this is not true for the smallest $2a-1$ elements. Let $u \in W$ be a monomial such that the $(2a -1)$-st smallest element in $F_t (u)$ is minimal among these elements in the sets $F_t (w)$ with $w \in W$.
Write $u=x_{i_1}\cdots x_{i_s}$ with $1\leq i_1\leq \cdots \leq i_s$. 

(II.1) We claim that $\deg u\geq a$. Suppose not, then since $a \le t$ it follows that the $(2a -1)$-st smallest element in 
$F_t (u)$ is $m+2a-1$.  By assumption, there is a monomial  $v = x_{j_1} \cdots x_{j_r} \in W$ with  $j_1 \le \cdots \le j_r$ such that the $(2a -1)$\textsuperscript{st} smallest element in $F_t (v)$ is greater than the $(2a -1)$\textsuperscript{st} smallest element in $F_t (u)$. 
Furthermore, as $\delta\geq 3$, Lemma~\ref{lem:technical}(c) gives that $\deg v > t$, and so the $(2a -1)$\textsuperscript{st} smallest element in $F_t (v)$ is $j_a + 2a -2 < m + 2a -1$. This is a contradiction to the choice of $u$. Thus, we have shown $\deg u \ge a$.

(II.2) Let $v=x_{j_1} \cdots x_{j_r}\in W$ be as in (III.1). We are going to show that $\deg(v)\geq a$. If $r \le a-2$, then the  $(2a -3)$\textsuperscript{rd} smallest 
element in $F_t (v)$ is $m + 2a-3$. By the choice of $a$, this number is equal to $i_{a-1} + 2a - 4 < m+2a -3$. This is a contradiction.  
Suppose $r = a-1$. Then the choice of $a$ 
gives that every monomial in $W$ is divisible by $v$. Since $a \le t$, we infer that 
\begin{equation}
     \label{eq:subset for u}
i_a + 2a -2, i_a + 2a -1 \in F_t (u). 
\end{equation}
By the choice of $v$ we also know $i_a + 2a -2 \notin F_t (v)$. Since $F_t (u) \cap F_t (v)$ is a ridge this forces that $i_a + 2a -1$ is the $(2a -1)$\textsuperscript{st} smallest element in $F_t (v)$. On the other hand  this element has to be equal to $m + 2a -1$, which yields $i_a = m$. It follows from the definitions of $u$ and $a$ that every element in $W$ is of the form $v x_m^e$ for some integer $e \ge 0$. However, the intersection of the images of three distinct such 
monomials under the map $F_t$ does not give a ridge.  This contradiction shows that we must have  
$r \ge a$.

(II.3) Let $w = x_{k_1} \cdots x_{k_p} \in W$, where $k_1 \le \cdots \le k_p$, such that the $(2a-1)$\textsuperscript{st} element of $F_t(w)$ equals  $i_a+2a-2$, i.e., the $(2a-1)$\textsuperscript{st} element of $F_t(u)$.  Suppose $p = \deg w < a$.  Since $a \le t$, the $(2a -1)$\textsuperscript{st}  smallest element in $F_t (w)$ is then $m+2a-1 > i_a + 2a -2$. This contradicts the choice of $w$ and we must have $p\geq a$. \smallskip
 
(II.4) We will finally show that $q_{2t}=q_{2t-1}+1$. Let $v=x_{j_1} \cdots x_{j_r}\in W$ as in (II.2). By the above $r\geq a$ and by the choice of $u$, it must hold that $i_a < j_a$. If $j_a \ge i_a + 2$, then neither $i_a + 2a -2$ nor  $i_a + 2a -1$ lies in $F_t (v)$ but they both lie in $F_t(u)$ (see \eqref{eq:subset for u}). This gives a contradiction to the fact that $F_t (u) \cap F_t (v)$ is  a  ridge. Hence we get
\begin{equation}
     \label{eq:ja}
j_a = i_a + 1. 
\end{equation}
If $w= x_{k_1} \cdots x_{k_p}\in W\setminus\{u,v\}$ is any other monomial, then the same reasoning shows that $k_a\in \{i_a,i_a+1\}$.  We consider two cases. \smallskip 

{\sf Case 2a.} Assume $k_a = i_a$.  Then \eqref{eq:ja} combined with Lemma~\ref{lem:technical}(b)(i) give that $\deg u\geq \deg v$ and $\deg w \geq \deg v$. Applying Lemma~\ref{lem:technical}(c) we conclude that the degrees of $u$ and $w$ are greater than $t$. Suppose $i_{a+1} > i_a$. Since $j_a = i_a + 1$, Lemma~\ref{lem:technical}(b)(ii) yields $x_{i_a + 1} \cdot u = x_{i_a} \cdot v$. If now $k_{a+1} > k_a$ we also get $x_{i_a + 1} \cdot w = x_{i_a} \cdot v$, and so $u = w$, a contradiction.  This forces $k_{a+1} = k_a = i_a$. Using also the assumption $i_{a+1} > i_a$ we conclude that 
\[
F_t (u) \cap F_t (v) \cap F_t (w) \subset F_t (w) \setminus \{ i_a + 2a-2, i_a + 2a \}. 
\]
Hence the left-hand side cannot be a ridge. This contradiction implies $i_{a+1} = i_a$. As $i_a + 2a -2 \notin F_t (v)$, this yields
\[
\bigcup_{\ell=1}^{a-1}\{i_\ell+2(\ell-1),i_\ell+2\ell-1\} \cup \{i_a + 2a -1, i_a + 2a \} \subset F_t (u) \cap F_t (v). 
\]
If $a = t$ this gives $q_{2t} = i_a + 2a = q_{2t-1} + 1$, as desired. 

Assume $a < t$. Using that $i_{a+1}+2a+1=i_a + 2a + 1 = j_a + 2a\in F_t (u)$ and $i_a+2a-1\in F_t(u)\setminus F_t(v)$ and the fact that $F_t(u)\cap F_t(v)$ is a ridge, we conclude that $j_a + 2 a \in F_t (v)$ and thus $\deg v > a$ and $j_{a+1} = j_a$. Moreover, by symmetry, the same reasoning as above shows that $k_{a+1} = k_a$. Hence in Case 2a, we have shown that $i_a = k_a = j_a - 1$ and $a < t$ imply $i_{a+1} = i_a = k_a = k_{a+1} $ and $j_{a+1} = j_a = i_a + 1$.  If $a+1 < t$ we repeat the argument to conclude that we must have $i_t  = \cdots = i_a = k_a = \cdots = k_t$ and $j_t = \cdots = j_a = i_a + 1$. Combined with $i_a + 2a -2 \notin F_t (v)$, this implies
\[
\bigcup_{\ell=1}^{a-1}\{i_\ell+2(\ell-1),i_\ell+2\ell-1\} \cup [i_a + 2a -1, i_a + 2t] \subset F_t (u) \cap F_t (v). 
\]
If follows that  $q_{2t} = i_a + 2t = q_{2t-1} + 1$, as desired. 
\smallskip 

{\sf Case 2b.} Assume $k_a = i_a + 1$. As in Case 2a, we have $\deg u \ge t+1 \ge a+1$.  If $i_{a+1} > i_a$, then Lemma~\ref{lem:technical}(b)(ii) gives $x_{i_a + 1} \cdot u = x_{i_a} \cdot v$ and $x_{i_a + 1} \cdot u = x_{i_a} \cdot w$, and so $v = w$, a contradiction. Hence, we must have $i_{a+1} = i_a$. It follows that 
\[
\bigcup_{\ell=1}^{a-1}\{i_\ell+2(\ell-1),i_\ell+2\ell-1\} \cup \{i_a + 2a -1, i_a + 2a \} \subset F_t (u) \cap F_t (v). 
\]
If $a = t$, then we deduce that $q_{2t} = i_a + 2a = q_{2t-1} + 1$, as desired. If $a < t$, we get 
$[i_a + 2a-2, i_a + 2a+1] \subset F_t (u)$. As $i_a + 2a - 2$ is neither in $F_t (v)$ nor in $F_t (w)$, this forces 
$i_a + 2a + 1 = j_a  + 2a = k_a + 2a \in F_t (v) \cap F_t (w)$. It follows that $j_{a+1} = j_a = k_a = k_{a+1}$. Hence we have 
shown that $k_a = j_a = i_a + 1$ and $a < t$ imply $j_{a+1} = j_a = k_a = k_{a+1} = i_a + 1$ and $i_{a+1} = i_a$. If $a+1 < t$ 
we repeat the argument to conclude that $j_t  = \cdots = j_a = k_a = \cdots = k_t = i_a +1$ and $i_t = \cdots = i_a$. This implies  as in Case 2a the desired equality $q_{2t} = i_a + 2t = q_{2t-1} + 1$. \smallskip 

Summarizing, we have proven that every ridge 
\[
F_t (u_1) \cap \cdots \cap F_t (u_{\delta}) = \{q_1,\ldots,q_{d+t-1}\}  = R
\]
of degree $\delta \ge 3$ satisfies $q_{2t} = q_{2t-1} + 1$. We are now going to show that, for every $j = 1,\ldots,\delta$, one has  
\begin{equation}
    \label{eq:smaller ridge}
\tilde{R} = \{q_1,\ldots,q_{2t-1}, q_{2t+1}-1,\ldots, q_{d+t-1}-1 \}  \subset F_{t-1} (u_j). 
\end{equation}
To this end fix some $u_j \in W$. Since $R$ is a ridge contained in $F_t (u_j)$, we get $F_t (u_j) = R \cup \{p\}$ for some $p$. If $p > q_{2t}$ Lemma~\ref{lem:technical}(a) yields $F_{t-1}(u_j) = \tilde{R} \cup \{p-1\}$, whereas we get $F_{t-1}(u_j) = \tilde{R} \cup \{q_{2t}-1\}$ if  $ q_{2t -1} < p < q_{2t}$.  Assume $p < q_{2t-1}$. Then we obtain 
\[
F_{t-1} (u_j) = \{q_1,\ldots,q_{2t-2}, p\} \cup \{q_{2t}-1,\ldots, q_{d+t-1}-1 \}.
\]
Since $q_{2t} -1 = q_{2t-1}$ it follows that $F_{t-1}(u_j) = \tilde{R} \cup \{p\}$. This completes the 
argument for establishing Containment~\eqref{eq:smaller ridge}. It shows that every ridge $R$ of 
$\Delta_t (U)$ whose degree is at least three gives rise to a ridge $\tilde{R}$ of $\Delta_{t-1} (U)$ 
whose degree is at least as big as the one of $R$. Moreover, the fact that for such ridges of $\Delta_t (U)$ the $(2t)$\textsuperscript{th} smallest element is determined by the $(2t-1)$\textsuperscript{st} smallest element implies that under the above correspondence no two distinct ridges of $\Delta_t (U)$ with degrees $\ge 3$ give rise to the same ridge of $\Delta_{t-1} (U)$. We conclude that 
\[
f_{d+t-2} (\Delta_t (U)) \cdot D (\Delta_t (U)) \le f_{d+t-3} (\Delta_{t-1} (U)) \cdot D (\Delta_{t-1} (U)).
\]
Since $f_{d+t-2} (\Delta_t (U))  > f_{d+t-3} (\Delta_{t-1} (U))$ our assertions follows. 
\end{proof}

\begin{remark}
Using similar arguments as in the proof of \Cref{thm:sing index} it is possible to show that also any ridge of $\Delta_t(U)$ of degree $2$, where $1\leq t\leq d_{\max}(U)$, gives rise to a ridge of degree \emph{at least} $2$ in $\Delta_{t-1}(U)$. However, the arguments are a bit more involved and more cases have to be distinguished. Since this (stronger) statement is not needed in this article, we omit the proof.
\end{remark}

\section{Open questions}
We close this article with some open questions. \\

\Cref{thm:sing index} together with \Cref{cor:shifted are squeezed} allows us to interpret the process of passing from $\Delta_0(U)$ to $\Delta_{d_{\max}}(U)$ via the family of squeezed complexes $(\Delta_t(U))_t$, where $U$ is a shifted order ideal, as a procedure to gradually convert a shifted simplicial complex into a ball with the same $h$-vector (up to additional zeros). This is reminiscent of the process of resolving singularities of a variety by repeated blow-ups in algebraic geometry. The following natural question arises:

\begin{question}
Can the procedure that converts a shifted simplicial complex into a pseudomanifold (of higher dimension) with the same $h$-vector be generalized to other families of simplicial complexes? For instance, given a non-pure shifted simplicial complex, is there a construction that yields a (non-pure) pseudomanifold with the same $h$-vector and that recovers squeezed complexes in the pure case?
\end{question}

We would like to remark that since a shifted simplicial complex $\Delta$ is always sequentially Cohen-Macaulay (even non-pure shellable) \cite[Section 4.5]{Ka-02}, a first (and rather naive) approach to this question could be to first apply the construction provided by the proof of \Cref{cor:shifted are squeezed} to the pure skeleta of $\Delta$ and then somehow ``glue''  the resulting complexes together.
\smallskip

On the one hand, we have seen in \Cref{thm:max order manifold} that for a maximal order ideal $U\subset P(m)$, a squeezed complex $\Delta_t(U)$ is a ball precisely if $t\geq d_{\max}(U)$. This means that for maximal order ideals the only squeezed balls are already given by the original construction due to Kalai. On the other hand, if $U\subset P(m)$ is such that the monomials of highest degree have a common divisor (which necessarily has to be a power of the variable $x_m$), then \Cref{prop: common divisor} shows that $\Delta_{d_{\max}(U)-1}(U)$ is a ball. Motivated by this, we suggest the following problem for further studies:

\begin{problem}\label{prob:balls}
Find classes of shifted order ideals such that for every shifted order ideal $U$ in this class there is a positive integer $\ell<d_{\max}(U)$ such that all squeezed complexes $\Delta_t(U)$ are balls for $t\geq \ell$. (Note that once, $\Delta_\ell(U)$ is a ball, so is $\Delta_{t}(U)$ for $t\geq \ell$ by \Cref{thm:sing index}.)
\end{problem}

Building on this problem, it is natural to study the boundaries of the squeezed balls $\Delta_t(U)$, where $U$ lies in a class of order ideals as defined by \Cref{prob:balls}. The first question that arises in this context is the following:

\begin{question}\label{question:new spheres}
Let $U$ be shifted order ideal such that $\Delta_\ell(U)$ is a ball for some $\ell<d_{\max}(U)$ (and so $\Delta_t(U)$ is a ball for $t\geq \ell$). Is there a shifted order ideal $\widetilde{U}$ with $d_{\max} (\widetilde{U}) \le \ell$ such that  the boundary of $\Delta_{d_{\max}(\widetilde{U})}(\widetilde{U})$ is isomorphic to the boundary of $\Delta_\ell(U)$, meaning that the boundary of $\Delta_\ell(U)$ is already a squeezed sphere in the sense of Kalai?	
\end{question}
There are only finitely many candidates for such an order ideal $\widetilde{U}$. Indeed, it follows from \cite[Proposition 5.2]{Ka} (see also \cite[Lemma 5.1]{Mu-10}) that the $h$-vector of $\partial \Delta_\ell(U)$ determines the $h$-vector of $\widetilde{U}$ and there are only finitely many shifted order ideals with the same $h$-vector. However, we do not even know the answer to this question if $U$ satisfies the assumption of \Cref{prop: common divisor}. In this case, even though~--~up to coning~--~$\Delta_{d_{\max}(U)-1}(U)$ is a squeezed complex in the sense of Kalai, it is not clear that also its boundary is a classical squeezed sphere. However, in all examples, we worked out, this was always the case. 

A negative answer to \Cref{question:new spheres} would imply that there are more spheres that can be obtained from our construction than there are squeezed spheres in the sense of Kalai. As a consequence, it would then be natural to study properties of those spheres such as strongly edge decomposability and Lefschetz properties (cf. \cite{Mu-10}).


\vspace*{1cm}







\begin{thebibliography}{99} 

\bibitem{ArHeHi}
A. Aramova, J. Herzog and T. Hibi, \emph{Ideals with stable Betti
numbers},  Adv. Math. {\bf 152} (2000), no. 1, 72--77. 

\bibitem{AHH}
A. Aramova, J. Herzog and T. Hibi, \emph{Shifting operations and graded Betti numbers}, 
J. Algebraic Combin. \textbf{12} (2000), no. 3,  207--222.

\bibitem{BNT}
E.\ Babson, I.\ Novik and  R.\ Thomas, {\em Reverse lexicographic and lexicographic shifting}, J.\ Algebraic\ Combin.\ {\bf 23} (2006), 107--123.

\bibitem{BMMNZ}
M.\ Boij,  J.\ Migliore, R.\ Mir\'o-Roig, U.\ Nagel, and  F.\ Zanello,  {\em On the shape of a pure $O$-sequence},
Mem.\ Amer.\ Math.\
  Soc.\ {\bf 218} (2012), no.\ 1024. 
  
\bibitem{BK} H.\ Brenner and A.\ Kaid, \emph{Syzygy bundles on $\mathbb P^2$ and the weak Lefschetz property}, Illinois J.\ Math.\ \textbf{51} (2007), no. 4, 1299--1308.  

\bibitem{CN}
D.\ Cook~II, U.\ Nagel,
    \emph{The weak Lefschetz property, monomial ideals, and lozenges},
    Illinois J.\ Math.\ \textbf{55} (2011), no.\ 1, 377--395.

\bibitem{CHMN}
D.\ Cook, B.\ Harbourne, J.\ Migliore and U.\ Nagel, 
\emph{Line arrangements and configurations of points with an unexpected geometric property},  Compositio Math. (to appear); arXiv:1602.02300.
  
\bibitem{DK}
G. Danaraj and V. Klee, \emph{Which spheres are shellable?},   Ann. Discrete Math. \textbf{2} (1978), 33--52. 

\bibitem{DIV} 
R.\ Di Gennaro, G.\ Ilardi and J.\ Vall\`es,  
\emph{Singular hypersurfaces characterizing the Lefschetz properties}, 
J.\ London Math.\ Soc. (2) {\bf 89} (2014), no.\ 1, 194--212 

\bibitem{EK}
S. Eliahou and M. Kervaire, \emph{Minimal resolutions of some
monomial ideals}. J. Algebra {\bf 129} (1990),  1--25.

\bibitem{FMS}
C. Francisco, J. Mermin, and J. Schweig, \emph{Borel generators}, J.\ Algebra \textbf{332}  (2011), 522--542.

\bibitem{HSS}
B.\  Harbourne, H.\ Schenck, and A.\ Seceleanu, {\em  Inverse systems, Gelfand-Tsetlin patterns 
and the weak Lefschetz property}, J.\ Lond. Math. Soc. (2) 84 (2011), no. 3, 712--730.

\bibitem{HMNW}
    T.\ Harima, J.\ Migliore, U.\ Nagel, and J.\ Watanabe,
    \emph{The weak and strong Lefschetz properties for Artinian $K$-algebras},
    J.\ Algebra \textbf{262} (2003), no.\ 1, 99--126. 
    
\bibitem{HMMNWW} 
T. Harima, T. Maeno, H. Morita, Y. Numata, A. Wachi, and J. Watanabe, 
\emph{Lefschetz properties}, Lecture Notes in Mathematics {\bf 2080}, Springer-Verlag, 2013.   
    
\bibitem{HH-Nagoya}
J. Herzog and T. Hibi, \emph{Componentwise linear ideals}, 
Nagoya Math.\ J.\ {\bf 153} (1999), 141--153.
    
    
\bibitem{HH}   
J. Herzog and T. Hibi, \emph{Monomial Ideals}, Graduate Texts in Mathematics \textbf{260}, Springer-Verlag, London, 2011. 
    
\bibitem{IK}
A.~Iarrobino and V.~Kanev,  {\em Power sums, {G}orenstein algebras, and determinantal loci},
 Lecture Notes in Mathematics {\bf 1721}, Springer-Verlag, Berlin, 1999.    

\bibitem{Ka} G.~Kalai, \emph{Many triangulated spheres}, Discrete Comput. Geom. \textbf{3} (1988), no. 1, 1--14.

\bibitem{Ka-02}
G.~Kalai, {\em Algebraic shifting}, In: Computational commutative algebra and combinatorics, 
Volume {\bf 33} of Adv. Stud. Pure Math., 121--163, 2002.
	
\bibitem{KRV} A.\ Kustin, H.\ Rahmati, A.\ Vraciu, \emph{The resolution of the bracket powers of the maximal ideal in a diagonal hypersurface ring}, J. Algebra 369 (2012), 256--321. 

\bibitem{MMO} E. Mezzetti, R. Mir\'o-Roig and G. Ottaviani, 
  \emph{Laplace Equations and the Weak Lefschetz Property}, Canad.\ J.\ Math.\  \textbf{65} (2013), 634--654.
  
\bibitem{Mu-07}
S.~Murai, \emph{Generic initial ideals and squeezed spheres},  Adv. Math. \textbf{214} (2007), no. 2, 701--729.

\bibitem{Mu-10}
S.~Murai, \emph{Algebraic shifting of strongly edge decomposable spheres}, J. Combin. Theory Ser. A  \textbf{117} (2010), no. 1, 1--16.

\bibitem{NR} 
U.~Nagel and T.~R\"omer, \emph{Criteria for componentwise linearity}, Comm.\ Algebra {\bf 43} (2015), 935--952. 

\bibitem{PB}
J.S. Provan and L.J. Billera, \emph{Decompositions of simplicial complexes related to diameters of convex polyhedra},  Math. of Operations Research \textbf{5} (1980), 576--594.

\bibitem{NT} 
U.~Nagel and B.~Trok, \emph{Interpolation and the weak Lefschetz property}, Preprint, 2018. 

\bibitem{St-faces} 
R.\ Stanley,  \emph{The number of faces of a simplicial convex polytope},
    Adv.\ Math. \textbf{35} (1980), 236--238. 
    
 \bibitem{Stanley-greenBook}
R.P.~Stanley, \emph{Combinatorics and commutative algebra}, Progress in Mathematics {\bf 41}, 1996.   

\bibitem{W} 
A. Wiebe, \emph{The Lefschetz property for componentwise linear ideals and Gotzmann ideals}, Comm.\ Algebra {\bf 32} (2004), 4601--4611. 


\end{thebibliography}
\end{document}